\documentclass[11pt,twoside]{amsart}
\usepackage[margin=3cm]{geometry}
\usepackage[colorlinks=false]{hyperref}
\usepackage[english]{babel}
\usepackage{graphicx,titling}
\usepackage{float}
\usepackage{amsmath,amsfonts,amssymb,amsthm}
\usepackage{lipsum}
\usepackage[T1]{fontenc}
\usepackage{fourier}
\usepackage{color}
\usepackage{esint}
\usepackage{caption}
\usepackage{ccicons}

\makeatletter
\def\blfootnote{\xdef\@thefnmark{}\@footnotetext}
\makeatother

\newcommand\ccnote{
    \blfootnote{Jie Zhou is supported by  NSFC 11721101 and NSFC 12531001.}
}
\allowdisplaybreaks
\usepackage[export]{adjustbox}
\numberwithin{equation}{section}
\usepackage{setspace}
\renewcommand{\le}{\leqslant}
\renewcommand{\leq}{\leqslant}
\renewcommand{\ge}{\geqslant}

\renewcommand{\mathbb}{\varmathbb}
\usepackage{fancyhdr}
\newtheorem{theorem}{Theorem}[section]
\newtheorem{lemma}[theorem]{Lemma}
\newtheorem{corollary}[theorem]{Corollary}
\newtheorem{proposition}[theorem]{Proposition}
\newtheorem{definition}[theorem]{Definition}
\newtheorem{remark}[theorem]{Remark}

\newtheorem{claim}[theorem]{Claim}

\usepackage{comment}
\usepackage{enumerate}
\usepackage{hyperref}
\usepackage[alphabetic, msc-links, backrefs]{amsrefs}

\usepackage{amsthm}

\allowdisplaybreaks

\input epsf
\def\begfig {
\begin{figure}
\small }
\def\endfig {
\normalsize
\end{figure}
}

\allowdisplaybreaks

\address{
\newline
Yuchen Bi\\Mathematical Institute\\ Department of Pure Mathematics\\ University of Freiburg\\ Ernst-Zermelo-Stra{\ss}e 1 Freiburg im Breisgau \\ D-79104\\ Germany
\newline
{\tt yuchen.bi@math.uni-freiburg.de}
\newline
 Jie Zhou,
School of Mathematical Sciences, Capital Normal University,
Beijing 100048, P.R. China.
\newline
{\tt Email:zhoujiemath@cnu.edu.cn}
}

\begin{document}

\thispagestyle{empty}

\ccnote

\vspace{1cm}


\begin{center}
\begin{huge}
\textit{Linear Quantitative Rigidity for Almost--CMC Surfaces}


\end{huge}
\end{center}

\vspace{1cm}


\begin{center}
\begin{minipage}[t]{.28\textwidth}
\begin{center}
{\large{\bf{Yuchen Bi,  Jie Zhou}}} \\
\vskip0.15cm
\end{center}
\end{minipage}
\end{center}

\vspace{1cm}

\vspace{1cm}
\vspace{1cm}
\noindent \textbf{Abstract.} \textit{We prove a quantitative rigidity result for almost constant mean
curvature spheres in $\mathbb{R}^3$.
Under a sub--two--sphere Willmore bound and a small $L^2$--CMC defect, we show
that an almost--CMC surface is close to the round sphere, with linear control
of the $W^{2,2}$--distance of the parametrization and the $L^\infty$--norm of
the conformal factor. An analogous statement holds under an a priori area bound below that
of two spheres.The proof relies on a linearized analysis around the sphere.  A previously established qualitative
rigidity result provides the initial closeness required to enter the
perturbative regime.  The estimate further extends to integral $2$--varifolds
of unit density using known regularity and density results.
}
\vskip0.3cm

\vspace{0.5cm}

\section{Introduction}

A classical theorem of Alexandrov asserts that an embedded closed hypersurface $\Sigma=\partial\Omega$
in Euclidean space $\mathbb{R}^{n+1}$  with constant mean curvature  must be a round sphere.
In recent years, a broad literature has developed  qualitative or quantitative 
rigidity results for Alexandrov-type characterizations of spheres under
various regularity assumptions and smallness of natural CMC deficits, with
conclusions ranging from closeness in weak geometric distances to stronger
graphical descriptions under additional non--degeneracy assumptions that
preclude bubbling.  We refer to \cite{CiMa,KM,CiVe,DeMaMiNe,FiZhang} and the
references therein for related results in this direction.

In this paper we study \emph{almost--CMC} surfaces in $\mathbb{R}^3$ and prove a
quantitative rigidity estimate with a \emph{linear rate} in the critical
$L^2$--defect $\|H-\overline H\|_{L^2}$.  From an analytic viewpoint a
natural way to measure deviation from the CMC condition is through
scale--invariant norms of $H-\overline H$; in dimension $n$ this leads to the
critical quantity $\|H-\overline H\|_{L^n}$.

In the two--dimensional setting in $\mathbb{R}^3$, a natural topology for such
critical problems is suggested by the conformal parametrization framework of
M\"uller--\v Sver\'ak\cite{MS95} and De~Lellis--M\"uller~\cite{dLMu,dLMu2}: one measures the
deviation from the round sphere through a conformal map and controls it in
$W^{2,2}$, together with an $L^\infty$ bound for the conformal factor.  Within
this framework, De~Lellis--M\"uller quantify nearness to a round sphere by the
$L^2$--smallness of the trace--free second fundamental form, whereas we quantify nearness to the CMC condition by
the oscillation of the scalar mean curvature $\|H-\overline H\|_{L^2}$.  Since small
oscillation of $H$ alone does not exclude bubbling, and bubbling is not
naturally captured by the conformal topology considered here, we impose in
addition an a priori two--sphere--threshold $L^2$--bound on $H$, namely
$\int_\Sigma |H|^2\,d\mu_g<32\pi$.

Quantitative results under scale--invariant assumptions in more general
settings, including higher dimensions and configurations with multiple
spherical components, are naturally formulated in topologies that are robust
under bubbling phenomena; see, for instance, the quantitative Alexandrov
theorem of Julin--Niinikoski~\cite{JulinNiinikoski23}.

As a starting point we use a qualitative stability result from~\cite{BiZhou22},
which shows that small $L^2$ CMC defect (together with the two--sphere--threshold
curvature bound) forces the surface into a perturbative $W^{2,2}$--conformal
neighborhood of the round sphere.  The main purpose of the present paper is to
upgrade the resulting modulus of continuity $\Psi(\delta)$ to a linear
bound.

\begin{theorem}[Qualitative rigidity for almost--CMC spheres {\cite[Theorem 1.2]{BiZhou22}}]
\label{thm:BZ22-qualitative}
For any $\alpha\in(0,\tfrac12)$ there exists $\delta_0=\delta_0(\alpha)>0$
with the following property.

Let $F:\Sigma\to\mathbb{R}^3$ be a smooth immersion of a closed connected surface,
let $g:=dF\otimes dF$ and $d\mu_g$ be the induced area measure, and let
$\vec N$ be a globally defined unit normal, with
scalar mean curvature $H:=\langle \vec H,\vec N\rangle$ and average
\[
\overline H := \fint_{\Sigma} H\,d\mu_g
       := \frac{1}{\mu_g(\Sigma)}\int_{\Sigma}H\,d\mu_g .
\]
Assume
\begin{equation}\label{eq:BZ22-assumptions}
\int_{\Sigma}|H-\overline H|^2\,d\mu_g \le \delta\le \delta_0,
\qquad
\int_{\Sigma}|H|^2\,d\mu_g \le 32\pi(1-\alpha).
\end{equation}
Then $\Sigma$ is homeomorphic to $\mathbb{S}^2$.

Moreover, after rescaling so that $\mu_g(\Sigma)=4\pi$,  and after composing $F$
with a suitable translation of $\mathbb{R}^3$, there exists a conformal
parametrization $ f:\mathbb{S}^2\to F(\Sigma)$ with
\[
d f\otimes d f= e^{2u}g_{\mathbb{S}^2},
\]
such that,
\begin{equation}\label{eq:BZ22-convergence}
\| f-\mathrm{id}_{\mathbb{S}^2}\|_{W^{2,2}(\mathbb{S}^2)}
 + \|u\|_{L^\infty(\mathbb{S}^2)}
 \le \Psi(\delta),
\end{equation}
where $\Psi(\delta)\to 0$ as $\delta\downarrow 0$.
\end{theorem}

We emphasize that our curvature assumption
\begin{equation}\label{eq:Willmore-below-two-spheres}
\int_{\Sigma} |H|^2\,d\mu_g \le 32\pi(1-\alpha)
\end{equation}
is primarily a \emph{single--bubble} condition, in the sense that it rules out
the formation of two well--separated spherical components.  We will refer to
\eqref{eq:Willmore-below-two-spheres} as a \emph{two--sphere--threshold} (or
\emph{sub--two--spheres}) Willmore bound.

Although Theorem~\ref{thm:BZ22-qualitative} is stated for smooth immersions, this assumption
has further geometric consequences.  By the classical result of Li--Yau\cite{LY}, any closed immersed surface in $\mathbb{R}^3$
with Willmore energy strictly below $8\pi$ is necessarily embedded.
In particular, under the curvature assumption
\eqref{eq:Willmore-below-two-spheres}, the immersion $F$ in
Theorem~\ref{thm:BZ22-qualitative} is an embedding.

Hence, without loss of generality, we may regard $F$ as the standard embedding
\[
\iota:\Sigma \hookrightarrow \mathbb{R}^3,
\]
and, by a mild abuse of notation, identify $\Sigma$ with its image
$\iota(\Sigma)\subset\mathbb{R}^3$. Since $\Sigma$ is a closed embedded surface,
there exists a precompact domain $\Omega\subset\mathbb{R}^3$ such that
$
\Sigma = \partial \Omega .$

In our quantitative analysis it is more convenient to normalize by the enclosed
volume $|\Omega|=\tfrac{4\pi}{3}$ rather than by the area
$\mu_g(\Sigma)=4\pi$.  This choice is immaterial, since in the perturbative
regime the two normalizations are equivalent up to a fixed scaling.

\medskip

We now state our main result.  As discussed above, it provides a
\emph{linear upgrade} of the qualitative rigidity estimate for almost--CMC
spheres, yielding quantitative control in the critical $L^2$--CMC defect.

\begin{theorem}[Linear quantitative rigidity for almost--CMC spheres]
\label{thm:main}
For any $\alpha\in(0,\tfrac12)$ there exist $\delta_0=\delta_0(\alpha)>0$ and
$C=C(\alpha)<\infty$ with the following property.

Let $\iota:\Sigma\to\mathbb{R}^3$ be a smooth immersion of a closed connected surface
satisfying \eqref{eq:BZ22-assumptions}.  After rescaling so that
$|\Omega|=\tfrac{4\pi}{3}$, up to a translation of $\mathbb{R}^3$,  there exists a conformal parametrization
$f:\mathbb{S}^2\to \Sigma$ with
\[
df\otimes df = e^{2u}g_{\mathbb{S}^2},
\]
such that
\begin{equation}\label{eq:linear-main-estimate}
\|h\|_{W^{2,2}(\mathbb{S}^2)} + \|u\|_{L^\infty(\mathbb{S}^2)}
\;\le\; C\,\|H-\overline H\|_{L^2(\Sigma)} ,
\end{equation}
where $h:=f-\mathrm{id}_{\mathbb{S}^2}$.
\end{theorem}

\textbf{Relation to perimeter--based formulations.}
Very recently, Julin--Morini--Oronzio--Spadaro~\cite{JulinMoriniOronzioSpadaro25}
proved a sharp quantitative Alexandrov inequality in $\mathbb{R}^3$: for every
$\delta_0>0$ there exists $C=C(\delta_0)>0$ such that any $C^2$--regular set
$E\subset\mathbb{R}^3$ with $|E|=|B_1|=\frac{4\pi}{3}$ and
\begin{equation}\label{eq:JMOS-perimeter-assumption}
P(E)\le 4\pi\,\sqrt[3]{2}-\delta_0
\end{equation}
satisfies
\begin{equation}\label{eq:JMOS-quant-Alex}
P(E)-P(B_1)\le C\,\|H_E-\overline H_E\|_{L^2(\partial E)}^2.
\end{equation}
Moreover, if $\partial E$ is smooth and the CMC defect is sufficiently small,
then $\partial E$ is diffeomorphic to $\mathbb{S}^2$.  The perimeter constraint
\eqref{eq:JMOS-perimeter-assumption} can be viewed as a natural non--degeneracy
assumption excluding the two--bubble threshold $4\pi\sqrt[3]{2}$.

For our purposes it is useful to relate \eqref{eq:JMOS-perimeter-assumption} to
the two--sphere--threshold Willmore bound \eqref{eq:Willmore-below-two-spheres}.
This implication can be deduced directly from the quantitative Alexandrov
theorem of Julin--Niinikoski ~\cite{JulinNiinikoski23}.  Since the argument is
short, we include it in Section~4 for completeness.  As a consequence, our main
theorem admits an equivalent formulation in which the Willmore assumption is
replaced by the area bound \eqref{eq:JMOS-perimeter-assumption}, at the expense
of assuming that the surface arises as the boundary of a smooth domain (since
without the Willmore bound embeddedness is not available a priori).

\begin{theorem}[Linear quantitative stability under an area bound]
\label{thm:main-area}
For any $\beta>0$ there exist
$\delta_0=\delta_0(\beta)>0$ and $C=C(\beta)<\infty$
with the following property.

Let $\Omega\subset\mathbb{R}^3$ be a bounded domain whose
 boundary $\Sigma$ is a smooth embedded surface, and
let $\iota:\Sigma\to\mathbb{R}^3$ denote the inclusion map.
After rescaling so that $|\Omega|=\tfrac{4\pi}{3}$, assume  that
\begin{equation}\label{eq:area-excludes-two-spheres}
\mu_g(\Sigma)\le 4\pi\,\sqrt[3]{2}-\beta,
\qquad
\|H-\overline H\|_{L^2(\Sigma)}\le \delta_0.
\end{equation}
Then there exists a conformal parametrization $f:\mathbb{S}^2\to \Sigma$ with
\[
df\otimes df=e^{2u}g_{\mathbb{S}^2},
\]
such that, after composing $f$ with a suitable translation and
writing $h:=f-\mathrm{id}_{\mathbb{S}^2}$, one has
\[
\|h\|_{W^{2,2}(\mathbb{S}^2)}+\|u\|_{L^\infty(\mathbb{S}^2)}
\;\le\; C\,\|H-\overline H\|_{L^2(\Sigma)}.
\]
\end{theorem}

With this reformulation, the assumptions of Theorem~\ref{thm:main-area} align
with those in \cite{JulinMoriniOronzioSpadaro25}.  In this sense, our result may
be viewed as a geometric counterpart of their quantitative Alexandrov
inequality.

\medskip
\textbf{A varifold formulation.}
Although Theorem~\ref{thm:main} is stated for smooth immersions, the same linear
control admits a convenient varifold formulation once one invokes density of
smooth embedded surfaces in the relevant class.  We record one such statement
below.

\begin{theorem}[Varifold stability for almost--CMC spheres]
\label{thm:varifold-almost-cmc}
Fix $\alpha\in(0,\tfrac12)$.  There exist $\delta_0=\delta_0(\alpha)>0$ and
$C=C(\alpha)<\infty$ with the following property.

Let $V=\underline{v}(\Sigma,1)$ be an integral $2$--varifold in
$\mathbb{R}^3$ with unit density and generalized mean curvature
$\vec H\in L^2(d\|V\|)$.  Assume that $\Sigma:=\mathrm{spt}\,\|V\|$ is compact and
that
\begin{equation}\label{eq:varifold-almost-cmc-assumptions}
\int_\Sigma |H-\overline H|^2\,d\|V\| \le \delta^2,
\qquad
\int_\Sigma |H|^2\,d\|V\| \le 32\pi(1-\alpha),
\qquad
0<\delta\le\delta_0,
\end{equation}
where $H:=\langle \vec H,\vec N\rangle$ is the scalar mean curvature defined
$\|V\|$--a.e.\ with respect to a measurable choice of unit normal, and
\[
\overline H:=\fint_\Sigma H\,d\|V\|
      :=\frac{1}{\|V\|(\mathbb{R}^3)}\int_\Sigma H\,d\|V\|.
\]

Then $\Sigma$ is homeomorphic to $\mathbb{S}^2$.  Moreover, after a rigid
motion of $\mathbb{R}^3$ and rescaling so that the enclosed volume equals
$\tfrac{4\pi}{3}$, there exists a homeomorphism
\[
f:\mathbb{S}^2\longrightarrow \Sigma
\]
with $f\in W^{2,2}(\mathbb{S}^2;\mathbb{R}^3)$ and a function
$u\in L^\infty(\mathbb{S}^2)$ such that
\[
df\otimes df = e^{2u} g_{\mathbb{S}^2}
\quad\text{a.e.\ on }\mathbb{S}^2,
\]
and the quantitative estimate
\begin{equation}\label{eq:varifold-linear-estimate}
\|f-\mathrm{id}_{\mathbb{S}^2}\|_{W^{2,2}(\mathbb{S}^2)}
+\|u\|_{L^\infty(\mathbb{S}^2)}
\;\le\; C\,\delta
\end{equation}
holds.
\end{theorem}

\smallskip
\noindent
\emph{Derivation from the smooth case.}
Under \eqref{eq:varifold-almost-cmc-assumptions}, one may approximate $\Sigma$ by a
sequence of smooth embedded surfaces $\Sigma_k$ in $W^{2,2}$ while preserving
unit density and the $L^2$--control of the mean curvature; this is provided,
for instance, for instance, by Corollary~5.7 in \cite{RS25} or Theorem~B.4 in \cite{BiZhou25b}.   Applying Theorem~\ref{thm:main} to
$\Sigma_k$
yields uniform linear bounds as in \eqref{eq:varifold-linear-estimate}, and a
compactness argument then passes the estimate to the varifold limit $V$.

\begin{remark}[A finite--perimeter variant]
\label{rem:finite-perimeter-variant}
In Theorem~\ref{thm:varifold-almost-cmc} the two--sphere--threshold Willmore
bound is used as a single--bubble assumption.  In view of the area formulation
(Theorem~\ref{thm:main-area}), one can replace this curvature bound by a
perimeter bound, provided the surface is known a priori to arise as the
boundary of a set of finite perimeter.

More precisely, let $E\subset\mathbb{R}^3$ be a bounded Caccioppoli set with
reduced boundary $\Sigma:=\partial^* E$, and assume that after
rescaling $|E|=\tfrac{4\pi}{3}$ one has
\[
\mu_g(\Sigma)\le 4\pi\,\sqrt[3]{2}-\beta,
\qquad
\int_{\Sigma}|H-\overline H|^2\,d\mu_g \le \delta^2,
\qquad
0<\delta\le \delta_0,
\]
with $\delta_0=\delta_0(\beta)$ as in
Theorem~\ref{thm:main-area}.  Then the same conclusion as in
Theorem~\ref{thm:main-area} holds.
\end{remark}

\medskip

\textbf{Strategy of the proof.}
Theorem~\ref{thm:BZ22-qualitative} reduces the quantitative problem to a
perturbative regime: after normalization and a suitable choice of rigid motion
(and conformal gauge), we may assume that $f$ is a $W^{2,2}$--conformal
immersion close to the standard embedding $f_0(x)=x$ on $\mathbb{S}^2$, with a
small conformal factor $u$.

The core of the paper consists of a linearized analysis around the round
sphere.  In Section~2 we derive a precise quadratic expansion of the geometric
energy governing the area excess and identify the leading quadratic form in
the normal component $z$ of the perturbation $h$.  By exploiting the spectral
properties of the spherical Laplacian together with the normalization and
orthogonality conditions, we obtain a coercive estimate that controls the
$W^{1,2}$--norm of the perturbation in terms of the CMC defect.

In Section~3 we estimate the deviation of the average mean curvature
$\overline H$ from the spherical value $2$ and use elliptic regularity to upgrade
the control to the full $W^{2,2}$--norm of $h$, as well as to derive an
$L^\infty$--bound for the conformal factor $u$.  This yields the desired linear
estimate \eqref{eq:linear-main-estimate} and completes the proof of
Theorem~\ref{thm:main}.

Finally, in Section~4 we give a short argument showing that the  two--sphere--threshold Willmore bound condition can equivalently be expressed in terms of an area
bound, and we derive the corresponding area formulation of our main theorem
(Theorem~\ref{thm:main-area}).

\section{$W^{1,2}$--Estimate}

Throughout this section we assume the smallness condition
\begin{equation}\label{eq:smallness-h-u}
\|f-f_0\|_{W^{2,2}(\mathbb{S}^2)}+\|u\|_{C^0(\mathbb{S}^2)}\ll 1,
\end{equation}
together with the volume normalization
\begin{equation}\label{volume normalization}
|\Omega|=\frac{4\pi}{3},
\end{equation}
which have been discussed in the
Introduction.

We normalize the parametrization of $f$ by minimizing its $L^2$--distance to
the standard embedding $f_0(x)=x$ in the translational and
M\"obius parameters. For $a\in\mathbb{R}^3$ and a M\"obius transformation
$\varphi:\mathbb{S}^2\to\mathbb{S}^2$, define
\[
\mathcal{E}(a,\varphi)
 := \int_{\mathbb{S}^2} |f\circ\varphi + a - f_0|^2\,dx.
\]

\begin{lemma}[Existence of a minimizer]\label{lem:partial-min}
Under the smallness assumption \eqref{eq:smallness-h-u} there exist
$\varphi_0\in \mathrm{Mob}(\mathbb{S}^2)$ and $a_0\in\mathbb{R}^3$ such that
$(a_0,\varphi_0)$ minimizes $\mathcal{E}(a,\varphi)$ among all translations
$a\in\mathbb{R}^3$ and all M\"obius transformations
$\varphi\in \mathrm{Mob}(\mathbb{S}^2)$.

Moreover, there exists $C>0$ such that,
\begin{equation}\label{eq:a0phi0-small}
|a_0|\le C\,\|f-f_0\|_{L^\infty(\mathbb{S}^2)},\qquad
\|\varphi_0-\mathrm{Id}\|_{L^2(\mathbb{S}^2)}\le C\,\|f-f_0\|_{L^\infty(\mathbb{S}^2)} .
\end{equation}
\end{lemma}

\begin{proof}
Set $\delta_\infty:=\|f-f_0\|_{L^\infty(\mathbb{S}^2)}$, which is small by
\eqref{eq:smallness-h-u} and Sobolev embedding.

We first compare $\mathcal{E}$ with the auxiliary quantity
\[
G(a,\varphi):=\int_{\mathbb{S}^2}|\varphi+a-f_0|^2\,dx,
\qquad a\in\mathbb{R}^3,\ \varphi\in\mathrm{Mob}(\mathbb{S}^2).
\]
Using $f_0(x)=x$ and the identity
\[
f\circ\varphi+a-f_0=(f-f_0)\circ\varphi+(\varphi+a-f_0),
\]
we obtain by the triangle inequality
\[
\|f\circ\varphi+a-f_0\|_{L^2}
\ge \|\varphi+a-f_0\|_{L^2}-\|(f-f_0)\circ\varphi\|_{L^2}
\ge \|\varphi+a-f_0\|_{L^2}-\sqrt{4\pi}\,\delta_\infty,
\]
and therefore
\[
\mathcal{E}(a,\varphi)=\|f\circ\varphi+a-f_0\|_{L^2}^2
\ge\Bigl(\|\varphi+a-f_0\|_{L^2}-\sqrt{4\pi}\,\delta_\infty\Bigr)^2.
\]
Since $(X-Y)^2\ge \frac12 X^2-Y^2$, this implies the rough bound
\begin{equation}\label{eq:E-vs-G-natural}
\mathcal{E}(a,\varphi)\ge \frac12\,G(a,\varphi)-4\pi\,\delta_\infty^2.
\end{equation}

Next, $\mathcal{E}(a,\varphi)\to\infty$ as $|a|\to\infty$, uniformly in $\varphi$.
Indeed, for $\delta_\infty\le1$ we have the pointwise estimate
$|f\circ\varphi-f_0|\le |f\circ\varphi-f_0\circ\varphi|+|f_0\circ\varphi-f_0|
\le \delta_\infty+2\le 3$, hence
\[
\mathcal{E}(a,\varphi)
=\int_{\mathbb{S}^2}|a+(f\circ\varphi-f_0)|^2\,dx
\ge 4\pi|a|^2-24\pi|a|.
\]
It follows that any minimizing sequence must remain in a bounded region
$\{|a|\le A\}$ for some $A<\infty$.

For the M\"obius parameter, we use the standard decomposition
$\varphi=R\circ\phi_v$ with $R\in SO(3)$ and $v\in B^3=\{v\in\mathbb{R}^3:|v|<1\}$,
where
\[
\phi_v(x):=\frac{(1-|v|^2)x+2(1+v\cdot x)\,v}{1+|v|^2+2v\cdot x}\in\mathbb{S}^2.
\]
If $v=r_i p_i$ with $p_i\in\mathbb{S}^2$ fixed and $r_i\uparrow1$, then after passing to a subsequence, we can assume $p_i\to p$ and $a_i\to a$, and
$\phi_{r_i p_i}(\cdot)\to p\in \mathbb{S}^2$ almost everywhere and hence in $L^2(\mathbb{S}^2)$.

As a result,
\begin{align*}
\int_{\mathbb{S}^2}|\phi_{r_i p_i}+a_i-x|^2\,dx\to \int_{\mathbb{S}^2}|p+a-x|^2dx=\int_{\mathbb{S}^2}(|p+a|^2+1)dx\ge 4\pi.
\end{align*}
Since $SO(3)$ is compact and rotations preserve the $L^2$--norm, the same
conclusion holds for $\varphi=R\circ\phi_v$. In particular, there exists
$\eta\in(0,1)$ and $c_*>0$ such that
\begin{equation}\label{eq:G-boundary-barrier-natural}
\inf_{a\in\mathbb{R}^3}G(a,\varphi)\ge c_*
\qquad\text{whenever }\varphi=R\circ\phi_v\text{ with }|v|\ge 1-\eta.
\end{equation}
Choosing $\delta_\infty$ so small that $4\pi\delta_\infty^2\le \frac14 c_*$, the
comparison \eqref{eq:E-vs-G-natural} yields
\[
\inf_{a\in\mathbb{R}^3}\mathcal{E}(a,\varphi)\ge \frac12c_*-4\pi\delta_\infty^2
\ge \frac14c_*
\qquad\text{for all }\varphi\text{ with }|v|\ge 1-\eta,
\]
whereas
\[
\mathcal{E}(0,\mathrm{Id})=\int_{\mathbb{S}^2}|f-f_0|^2\,dx \le 4\pi\delta_\infty^2
\le \frac14c_*.
\]
Consequently, a minimizing sequence cannot degenerate to the boundary
$|v|\to1$, and together with the coercivity in $a$ this shows that it suffices
to minimize $\mathcal{E}$ over the compact set
\[
K:=\{\,a\in\mathbb{R}^3:|a|\le A\,\}\times
\{\,R\circ\phi_v : R\in SO(3),\ |v|\le 1-\eta\,\}.
\]
The map $(a,\varphi)\mapsto \mathcal{E}(a,\varphi)$ is continuous on $K$, hence
it attains its minimum at some $(a_0,\varphi_0)\in K$. By construction,
$(a_0,\varphi_0)$ is also a global minimizer among all translations and all
M\"obius transformations.

For the estimate, minimality gives
\[
\mathcal{E}(a_0,\varphi_0)\le \mathcal{E}(0,\mathrm{Id})\le 4\pi\delta_\infty^2.
\]
Combining this with \eqref{eq:E-vs-G-natural} yields
\begin{equation}\label{eq:G-small-natural}
G(a_0,\varphi_0)\le C\,\delta_\infty^2.
\end{equation}

We next bound $a_0$ quantitatively. Since $\varphi_0(x)\in\mathbb{S}^2$ for all
$x$, we have
\[
|\varphi_0(x)+a_0-x|
=|\varphi_0(x)-(x-a_0)|
\ge \mathrm{dist}(x-a_0,\mathbb{S}^2)
= \bigl||x-a_0|-1\bigr|.
\]
Consequently,
\[
G(a_0,\varphi_0)\ge
\int_{\mathbb{S}^2}\bigl(\,||x-a_0|-1|\,\bigr)^2\,dx
=:J(a_0).
\]
Recall from the coercivity in the translation parameter that any minimizer
satisfies $|a_0|\le A$ for some fixed $A<\infty$. For such $a$ we estimate
\[
\bigl||x-a|-1\bigr|
= \frac{\bigl||x-a|^2-1\bigr|}{|x-a|+1}
= \frac{\bigl||a|^2-2a\cdot x\bigr|}{|x-a|+1}
\ge \frac{\bigl||a|^2-2a\cdot x\bigr|}{|a|+2}
\ge \frac{\bigl||a|^2-2a\cdot x\bigr|}{A+2}.
\]
Squaring and integrating yields
\[
J(a)\ge \frac{1}{(A+2)^2}\int_{\mathbb{S}^2}\bigl(|a|^2-2a\cdot x\bigr)^2\,dx.
\]
Using $\int_{\mathbb{S}^2}x\,dx=0$ and
$\int_{\mathbb{S}^2}(a\cdot x)^2\,dx=\frac{4\pi}{3}|a|^2$, we compute
\[
\int_{\mathbb{S}^2}\bigl(|a|^2-2a\cdot x\bigr)^2\,dx
=4\pi|a|^4+\frac{16\pi}{3}|a|^2
\ge \frac{16\pi}{3}|a|^2,
\]
and therefore
\begin{equation}\label{eq:J-coercive-A}
J(a)\ge \frac{16\pi}{3(A+2)^2}\,|a|^2
\qquad\text{for all }|a|\le A.
\end{equation}
Applying \eqref{eq:J-coercive-A} to $a_0$ and combining with
$J(a_0)\le G(a_0,\varphi_0)$ and \eqref{eq:G-small-natural}, we conclude
\[
|a_0|^2 \le C\,J(a_0)\le C\,G(a_0,\varphi_0)\le C\,\delta_\infty^2,
\qquad\text{hence}\qquad |a_0|\le C\,\delta_\infty.
\]
Finally,
\[
\|\varphi_0-\mathrm{Id}\|_{L^2(\mathbb{S}^2)}
=\|\varphi_0-f_0\|_{L^2(\mathbb{S}^2)}
\le \|\varphi_0+a_0-f_0\|_{L^2(\mathbb{S}^2)} + \|a_0\|_{L^2(\mathbb{S}^2)}
\le G(a_0,\varphi_0)^{1/2} + \sqrt{4\pi}\,|a_0|
\le C\,\delta_\infty,
\]
which proves \eqref{eq:a0phi0-small}.
\end{proof}

We now replace $f$ by the normalized map
\[
\tilde{f} := f\circ\varphi_0 + a_0,
\qquad h:=\tilde{f}-f_0.
\]
By Lemma~\ref{lem:partial-min} and $\|\varphi_0-\mathrm{Id}\|_{C^2}\le C\|\varphi_0-\mathrm{Id}\|_{L^2}$, we know
\begin{align*}
\|\tilde{f}-f\|_{W^{2,2}}=\|(f-f_0)\circ \varphi_0+(\varphi_0+a)\|_{W^{2,2}}\le C \|f-f_0\|_{W^{2,2}}
\end{align*}
and $
e^{2\tilde{u}}g_{\mathbb{S}^2}=d\tilde{f}\otimes d\tilde{f}=e^{2u(\varphi_0(x))}\mathrm{Jac}_{\varphi_0}(x)g_{\mathbb{S}^2}$
implies
\begin{align*}
\|\tilde{u}\|_{L^\infty}=\|u(\varphi_0(x))+\frac{1}{2}\ln{\mathrm{Jac}_{\varphi_0}(x)}\|_{L^\infty}\le \|u\|_{L^\infty}+C\|\varphi_{0}-\mathrm{Id}\|_{C^1}\le C(\|u\|_{L^\infty}+\|f-f_0\|_{W^{2,2}}).
\end{align*}

So, $\tilde{f}$ still satisfies the smallness condition \eqref{eq:smallness-h-u} and volume renormalization condition \eqref{volume normalization}.  From now on, we will still use the notation $f$ to denote normalized map $\tilde{f}$. Then, by construction, the point $(a,\varphi)=(0,\mathrm{Id})$ is a \emph{critical point}
of $\mathcal{E}$ with respect to variations in $a$ and in $\varphi$
generated by infinitesimal M\"obius transformations.
This yields the following orthogonality identities.

\begin{lemma}[Orthogonality identities]\label{lem:norm-orth}
For the normalized map $f$, the deviation $h=f-f_0$ satisfies
\begin{align}
\int_{\mathbb{S}^2} h\,dx &= 0,
   \label{eq:orth-const}\\
\int_{\mathbb{S}^2} h\cdot f_*(X)\,dx &= 0
   \qquad \forall\ \text{conformal Killing fields }X.
   \label{eq:orth-conf}
\end{align}
\end{lemma}

\begin{proof}
 Varying $a_t=t\xi$,
\[
0=\frac{d}{dt}\bigg|_{t=0}\mathcal{E}(a_t,\mathrm{Id})
 = 2\int_{\mathbb{S}^2} h\cdot\xi\,dx,
\]
so $\int h=0$. If $\varphi_t$ is generated by a conformal Killing field $X$, then
\(\frac{d}{dt}|_{t=0}(f\circ\varphi_t)=f_*(X)\), hence
\[
0=\frac{d}{dt}\bigg|_{t=0}\mathcal{E}(0,\varphi_t)
 = 2\int h\cdot f_*(X)\,dx.
\]
\end{proof}

We next decompose
\[
h = v + z n, \qquad n=f_0.
\]

\begin{lemma}[Consequences of the normalization]\label{lem:norm-conseq}
For each coordinate function $x^i$ the following hold:
\begin{align}
\int_{\mathbb{S}^2} v\cdot\nabla x^i\,dx
 + \int_{\mathbb{S}^2} z\,x^i\,dx &= 0,
 \label{eq:relation-1}\\[0.3em]
\bigl|\int_{\mathbb{S}^2} v\cdot Af_0\,dx\bigr|
 &\le C\,\|h\|_{W^{1,2}(\mathbb{S}^2)}^2,\qquad \forall A\in\mathfrak{so}(3),
\label{eq:orth-rot}\\[0.3em]
\bigl|\int_{\mathbb{S}^2} v\cdot\nabla x^i\,dx\bigr|
 &\le C\,\|h\|_{W^{1,2}(\mathbb{S}^2)}^2,
 \label{eq:relation-2}\\[0.3em]
\bigl|\int_{\mathbb{S}^2} z\,x^i\,dx\bigr|
 &\le C\,\|h\|_{W^{1,2}(\mathbb{S}^2)}^2.
 \label{eq:z-xi-bound}
\end{align}
\end{lemma}

\begin{proof}
Since $\int_{\mathbb{S}^2} h\,dx=0$ and $h=v+zn$ with $n=f_0$, we have
\[
\int_{\mathbb{S}^2} v\,dx + \int_{\mathbb{S}^2} z n\,dx = 0.
\]
Taking the Euclidean inner product with $e_i$ and using
$v\cdot e_i=v\cdot\nabla x^i$ and $n\cdot e_i=x^i$ gives
\[
\int_{\mathbb{S}^2} v\cdot\nabla x^i\,dx
 + \int_{\mathbb{S}^2} z\,x^i\,dx = 0,
\]
which is \eqref{eq:relation-1}.

For \eqref{eq:orth-rot}, fix $A\in\mathfrak{so}(3)$ and apply
\eqref{eq:orth-conf} with the conformal Killing field $X=Af_0$:
\[
0 = \int_{\mathbb{S}^2} h\cdot f_*(Af_0)\,dx.
\]
Since $f=f_0+h$ and $df=df_0+dh$, we have
\[
f_*(Af_0) = df(Af_0) = df_0(Af_0)+dh(Af_0) = Af_0 + dh(Af_0),
\]
because $df_0$ is the identity on $T\mathbb{S}^2$.  Thus
\[
0 = \int_{\mathbb{S}^2} h\cdot Af_0\,dx
    + \int_{\mathbb{S}^2} h\cdot dh(Af_0)\,dx.
\]
Decomposing $h=v+zn$ and using $n\cdot Af_0=0$ (since $Af_0$ is tangent),
we obtain
\[
\int_{\mathbb{S}^2} v\cdot Af_0\,dx
 = -\int_{\mathbb{S}^2} h\cdot dh(Af_0)\,dx.
\]
The pointwise bound $|dh(Af_0)|\le C|\nabla h|$ yields
\[
\bigl|\int_{\mathbb{S}^2} v\cdot Af_0\,dx\bigr|
 \le C\int_{\mathbb{S}^2} |h|\,|\nabla h|\,dx
 \le C\,\|h\|_{L^2(\mathbb{S}^2)}\,\|\nabla h\|_{L^2(\mathbb{S}^2)}
 \le C\,\|h\|_{W^{1,2}(\mathbb{S}^2)}^2.
\]
For \eqref{eq:relation-2}, we apply \eqref{eq:orth-conf} with
$X=\nabla x^i$:
\[
0 = \int_{\mathbb{S}^2} h\cdot f_*(\nabla x^i)\,dx
  = \int_{\mathbb{S}^2} h\cdot\bigl(\nabla x^i + dh(\nabla x^i)\bigr)\,dx.
\]
Since $h=v+zn$ and $n\cdot\nabla x^i=0$, we get
\[
\int_{\mathbb{S}^2} v\cdot\nabla x^i\,dx
 = -\int_{\mathbb{S}^2} h\cdot dh(\nabla x^i)\,dx.
\]
Using $|\nabla x^i|\le C$ and $|dh(\nabla x^i)|\le C|\nabla h|$,
\[
\bigl|\int_{\mathbb{S}^2} v\cdot\nabla x^i\,dx\bigr|
 \le C\int_{\mathbb{S}^2} |h|\,|\nabla h|\,dx
 \le C\,\|h\|_{L^2(\mathbb{S}^2)}\,\|\nabla h\|_{L^2(\mathbb{S}^2)}
 \le C\,\|h\|_{W^{1,2}(\mathbb{S}^2)}^2.
\]
Finally, combining \eqref{eq:relation-1} and \eqref{eq:relation-2} gives
\[
\bigl|\int_{\mathbb{S}^2} z\,x^i\,dx\bigr|
 = \bigl|\int_{\mathbb{S}^2} v\cdot\nabla x^i\,dx\bigr|
 \le C\,\|h\|_{W^{1,2}(\mathbb{S}^2)}^2,
\]
which is \eqref{eq:z-xi-bound}.
\end{proof}

We now recall the Cauchy--Riemann operator that encodes the trace-free
part of the conformality condition.
The conformality of $f$ reads
\[
(df_0 + dh)\otimes(df_0 + dh) = e^{2u} g_{\mathbb{S}^2},
\]
so, using $df_0\otimes df_0 = g_{\mathbb{S}^2}$, we obtain
\begin{equation}\label{eq:conf-relation}
\,df_0\otimes dh+dh\otimes df_0 + dh\otimes dh = (e^{2u}-1)\,g_{\mathbb{S}^2}.
\end{equation}
Taking the trace-free part of \eqref{eq:conf-relation}, we obtain
\[
\bigl(\,df_0\otimes dh+ dh\otimes df_0 + dh\otimes dh\bigr)_0 = 0,
\]
where by $(T)_0$ we mean the trace free part of a tensor $T$.
Let $(e_1,e_2)$ be a local $g_{\mathbb{S}^2}$-orthonormal frame, and write
\[
h_i := dh(e_i)\in\mathbb{R}^3,\qquad
v_{i,j}:=\langle \nabla_{e_j} v, e_i\rangle .
\]
Since $h=v+zn$, $df
_0\otimes ndz=f_0df_0\otimes dz=0$ and $dn(e_i)=e_i=df_0(e_i)$ on the unit sphere, the normal component $zn$
contributes only a pure-trace term to $(df_0\otimes dh)_0$, hence it drops out
after taking the trace-free part.  In these coordinates the trace-free
conformality conditions are the two scalar equations
\begin{equation}\label{eq:CR-components}
\begin{aligned}
v_{1,1}-v_{2,2} &= -\frac12\bigl(|h_1|^2-|h_2|^2\bigr),\\
v_{1,2}+v_{2,1} &= -\langle h_1,h_2\rangle .
\end{aligned}
\end{equation}
Equivalently, if we define the Cauchy--Riemann operator on vector fields by
\[
D v := \bigl(\mathcal{L}_vg_{\mathbb{S}^2}\bigr)_0 \in \Gamma(S^2_0T^*\mathbb{S}^2),
\qquad
(Dv)_{11}=v_{1,1}-v_{2,2},\ \ (Dv)_{12}=v_{1,2}+v_{2,1},
\]
then \eqref{eq:CR-components} can be written compactly as
\begin{equation}\label{eq:DvQ}
D v = Q(dh),
\end{equation}
where $Q(dh)$ is the trace-free quadratic form
\[
Q(dh):= -\,(dh\otimes dh)_0,
\quad\text{i.e.}\quad
Q_{11}=-\frac12\bigl(|h_1|^2-|h_2|^2\bigr),\ \ Q_{12}=-\langle h_1,h_2\rangle .
\]
In particular, $Q(dh)$ is quadratic in $dh$ and satisfies the pointwise bound
\begin{equation}\label{eq:Qbound}
|Q(dh)| \le C\,|dh|^2 \le C\,|\nabla h|^2 .
\end{equation}

If we view $\mathbb{S}^2$ as the Riemann sphere and identify
$T\mathbb{S}^2\simeq\mathcal{O}(2)$, then $D$ corresponds to the Dolbeault
operator
\[
\overline\partial : \Gamma(T^{1,0}\mathbb{S}^2)\longrightarrow
\Gamma(T^{1,0}\mathbb{S}^2\otimes\Lambda^{0,1}\mathbb{S}^2),
\]
which is surjective, and whose kernel is precisely the 6-dimensional space of conformal Killing fields. 

Standard elliptic theory for $D$
gives the following estimate.

\begin{lemma}[Elliptic estimate for the tangential part]\label{lem:v-Lp}
Let $1<p<\infty$.
There exists a constant $C_p<\infty$ such that the following holds:
if $v$ is the tangential part of $h=f-f_0$ as above, then
\begin{equation}\label{eq:v-W1p}
\|v\|_{W^{1,p}(\mathbb{S}^2)}
 \;\le\; C_p\Bigl(\|Q(dh)\|_{L^p(\mathbb{S}^2)}
                 + \|h\|_{W^{1,2}(\mathbb{S}^2)}^2\Bigr).
\end{equation}
In particular, using \eqref{eq:Qbound},
\begin{equation}\label{eq:v-Lp-quad}
\|v\|_{W^{1,p}(\mathbb{S}^2)}
 \;\le\; C_p\Bigl(\|\nabla h\|_{L^{2p}(\mathbb{S}^2)}^2
                 + \|h\|_{W^{1,2}(\mathbb{S}^2)}^2\Bigr).
\end{equation}
\end{lemma}

\begin{proof}
Decompose $v$ into its component orthogonal to $\ker D$ and its projection
onto $\ker D$:
\[
v = v_0 + v_{\ker},\qquad v_0\perp\ker D,\quad v_{\ker}\in\ker D.
\]
On the orthogonal complement,using \eqref{eq:DvQ}, there is an elliptic estimate
\[
\|v_0\|_{W^{1,p}} \le C_p \|D v_0\|_{L^p}
 = C_p \|D v\|_{L^p}
 = C_p \|Q(dh)\|_{L^p}.
\]

Next, $v_{\ker}$ lies in the $6$-dimensional space of conformal Killing
fields, spanned by $A f_0$ and $\nabla x^i$, where $A\in\mathfrak{so}(3)$ . So $v_{\ker}$ is a linear combination of $A f_0$ and
$\nabla x^i$.  The coefficients  are controlled by \eqref{eq:orth-rot} and
\eqref{eq:relation-2} , which show that
\[
\Bigl|\int_{\mathbb{S}^2} v\cdot\nabla x^i\,dx\Bigr|
 + \Bigl|\int_{\mathbb{S}^2} v\cdot Af_0\,dx\Bigr|
 \le C\,\|h\|_{W^{1,2}(\mathbb{S}^2)}^2.
\]
Since the $\nabla x^i$ and $A_if_0$ form a fixed basis , this implies
\[
\|v_{\ker}\|_{W^{1,p}(\mathbb{S}^2)} \le C\,\|h\|_{W^{1,2}(\mathbb{S}^2)}^2.
\]

Combining the two parts, we obtain
\[
\|v\|_{W^{1,p}(\mathbb{S}^2)}
 \le C_p\Bigl(\|Q(dh)\|_{L^p(\mathbb{S}^2)}
              + \|h\|_{W^{1,2}(\mathbb{S}^2)}^2\Bigr),
\]
which is \eqref{eq:v-W1p}.

Finally, by \eqref{eq:Qbound} we have
\[
\|Q(dh)\|_{L^p(\mathbb{S}^2)} \le C\,\||\nabla h|^2\|_{L^p}
 = C\,\|\nabla h\|_{L^{2p}(\mathbb{S}^2)}^2,
\]
and substituting this into \eqref{eq:v-W1p} yields \eqref{eq:v-Lp-quad}.
\end{proof}

 We write
\[
E[h]
 := \int_{\mathbb{S}^2}\bigl(|\nabla(f_0+h)|^2-|\nabla f_0|^2\bigr)\,dx
\]
for the Dirichlet energy difference.


\begin{lemma}\label{lem:energy-identities-c}
Under the above assumptions, the following identities hold:
\begin{align}
E[h]
 &= 4\int_{\mathbb{S}^2} z\,dx
    + \int_{\mathbb{S}^2} |\nabla h|^2\,dx, \label{eq:E-basic-c}\\[0.4em]
E[h]
 &= -2\int_{\mathbb{S}^2}\Delta(f_0+h)\cdot h\,dx
    - \int_{\mathbb{S}^2}|\nabla h|^2\,dx,\label{eq:E-Delta-c}
\end{align}
\end{lemma}

\begin{proof}
By definition,
\begin{align}\label{eq:E-basic-grad}
E[h]= \int_{\mathbb{S}^2}\bigl(|\nabla(f_0+h)|^2-|\nabla f_0|^2\bigr)\,dx= 2\int_{\mathbb{S}^2}\nabla f_0\cdot\nabla h\,dx
   + \int_{\mathbb{S}^2}|\nabla h|^2\,dx .
\end{align}
Noting $-\Delta f_0=2f_0$, we know
\begin{align*}
2\int_{\mathbb{S}^2}\nabla f_0\cdot\nabla h\,dx=-2\int_{\mathbb{S}^2}\Delta f_0\cdot hdx=4\int_{\mathbb{S}^2}f_0\cdot h dx=4\int_{\mathbb{S}^2}zdx.
\end{align*}
Substituting this into \eqref{eq:E-basic-grad} yields \eqref{eq:E-basic-c}.  Again by
\eqref{eq:E-basic-grad}, we know
\begin{align*}
E[h]=2\int_{\mathbb{S}^2}\nabla(f_0+h)\cdot \nabla h-\int_{\mathbb{S}^2}|\nabla h|^2dx=-2\int_{\mathbb{S}^2}\Delta(f_0+h)\cdot hdx-\int_{\mathbb{S}^2}|\nabla h|^2dx,
\end{align*}
which is \eqref{eq:E-Delta-c}.
\end{proof}

\begin{lemma}[Relation between the averages of $z$ and $z^2$]
\label{lem:z-vs-z2-volume}
Under the normalization and smallness assumption
$\delta:=\|h\|_{W^{2,2}(\mathbb{S}^2)}\ll1$, the scalar normal component
$z$ satisfies
\begin{equation}\label{eq:z-vs-z2-volume}
\biggl|\int_{\mathbb{S}^2} z\,dx + \int_{\mathbb{S}^2} z^2\,dx\biggr|
 \;\le\; C\,\delta^3.
\end{equation}
\end{lemma}

\begin{proof}
By the divergence theorem,
\[
\int_{\Sigma} f\cdot n_f\,d\mu_f = 3|\Omega|,\qquad
\int_{\mathbb{S}^2} f_0\cdot n_0\,d\mu_{f_0} = 3|B_1|,
\]
and by our volume normalization $|\Omega|=|B_1|$ these integrals coincide.
In local conformal coordinates we have
\[
f\cdot(f_1\times f_2)\,dy = f\cdot n_f\,d\mu_f,\qquad
f_0\cdot (f_0)_1\times (f_0)_2\,dy = f_0\cdot n_0\,d\mu_{f_0},
\]
so
\begin{equation}\label{eq:volume-diff-zero}
\int_{\mathbb{S}^2}\bigl(f\cdot f_1\times f_2
      - f_0\cdot (f_0)_1\times (f_0)_2\bigr)\,dx = 0.
\end{equation}

Fix a point and choose an oriented orthonormal frame
$(e_1,e_2,n)$ in $\mathbb{R}^3$ with
\[
e_1=(f_0)_1,\quad e_2=(f_0)_2,\quad n=e_1\times e_2=f_0.
\]
Write $h=v+zn$ with $v\perp n$, and denote $f_i:=\partial_i f$,
$h_i:=\partial_i h$, etc.  Then
\[
\begin{aligned}
f\cdot f_1\times f_2 - f_0\cdot (f_0)_1\times (f_0)_2
&= (f_0+h)\cdot(f_0+h)_1\times(f_0+h)_2 - n\cdot e_1\times e_2\\
&= h\cdot e_1\times e_2
   + n\cdot h_1\times e_2
   + n\cdot e_1\times h_2\\
&\quad + h\cdot h_1\times e_2
       + h\cdot e_1\times h_2
       + n\cdot h_1\times h_2
       + h\cdot h_1\times h_2.
\end{aligned}
\]
We analyse the terms one by one in the frame $(e_1,e_2,n)$.

Firstly, since $h=v+zn$ and $e_1\times e_2=n$, we have
  \[
  h\cdot e_1\times e_2 = h\cdot n = z.
  \]
For the other terms, we use $h=v+zn$ and the identities
  \[
  n\times e_1 = e_2,\qquad e_2\times n = e_1.
  \]
  For the first order term with derivative of $h$, the above identity implies $$n\cdot h_1\times e_2=h_1\cdot e_2\times n=h_1\cdot e_1 \quad \text { and } \quad h_2\cdot n\times e_1=h_2\cdot e_2.$$  Noting $\nabla z\cdot n=0$, this further implies
  \[
  n\cdot h_1\times e_2 + n\cdot e_1\times h_2 =\mathrm{div}h= \mathrm{div}\,v+2z.
  \]

  For the quadratic and cubic terms, expanding the cross products gives
  \[
  \begin{aligned}
  h\times h_1
  &= (v+zn)\times h_1
   = v\times h_1 + zn\times(v_1+z_1n+ze_1)
   = z^2 e_2 + R_1,\\
  h_2\times h
  &= h_2\times(v+zn)
   = h_2\times v + (v_2+z_2 n+ze_2)\times zn
   = z^2 e_1 + R_2,\\
  h_1\times h_2
  &= v_1\times h_2 + (z_1n+ze_1)\times(v_2+z_2n+ze_2)
   = -z\nabla z + z^2 n + R_3,
  \end{aligned}
  \]
  where the remainders are
  \begin{align*}
  R_1=v\times h_1+zn\times v_1\quad \quad R_2=h_2\times v+z\cdot v_2\times n\quad \text{ and } \quad R_3=v_1\times h_2+(z_1n+ze_1)\times v_2.
  \end{align*}
  where the remainders satisfy the pointwise bounds
  \[
  |R_1|+|R_2|+|R_3|
   \;\le\; C\bigl(|v|\,|\nabla h|+|\nabla v|(|h|+|\nabla h|)\bigr).
  \]
  Consequently,
  \[
  \begin{aligned}
  I &:= h\cdot h_1\times e_2
      = e_2\cdot(z^2 e_2 + R_1) = z^2 + O(|R_1|),\\
  II &:= h\cdot e_1\times h_2
      = e_1\cdot(z^2 e_1 + R_2) = z^2 + O(|R_2|),\\
  III &:= n\cdot h_1\times h_2
       = n\cdot(-z\nabla z + z^2 n + R_3)
       = z^2 + O(|R_3|),
  \end{aligned}
  \]
  and the fully cubic term $h\cdot h_1\times h_2$ is $O(|h|\,|\nabla h|^2)$.

Putting everything together, we obtain the pointwise expansion
\begin{equation}\label{eq:volume-expansion-pointwise}
f\cdot f_1\times f_2 - f_0\cdot (f_0)_1\times (f_0)_2
 = 3z + 3z^2 + \mathrm{div}\,v + R,
\end{equation}
where
\[
|R|\;\le\;C\bigl((|v|+|\nabla v|)(|h|+|\nabla h|) + |h|\,|\nabla h|^2\bigr).
\]

Integrating \eqref{eq:volume-expansion-pointwise} over $\mathbb{S}^2$ and
using \eqref{eq:volume-diff-zero} and $\int_{\mathbb{S}^2}\mathrm{div}\,v=0$,
we obtain
\[
\int_{\mathbb{S}^2} z\,dx + \int_{\mathbb{S}^2} z^2\,dx
 = -\frac{1}{3}\int_{\mathbb{S}^2} R\,dx.
\]
By Lemma~\ref{lem:v-Lp} (with $p=2$) we have
$\|v\|_{W^{1,2}}\le C\delta^2$, while $\|h\|_{L^\infty}$ and
$\|\nabla h\|_{L^2}$ are $O(\delta)$ by Sobolev embedding.  Thus
\[
\int_{\mathbb{S}^2}|R|
 \le C\int_{\mathbb{S}^2}\bigl((|v|+|\nabla v|)(|h|+|\nabla h|) + |h|\,|\nabla h|^2\bigr)
 \le C\delta^2\delta + C\delta\delta^2
 \le C\delta^3.
\]
This yields \eqref{eq:z-vs-z2-volume}.
\end{proof}

\begin{lemma}[Refined expansion of $|\nabla h|^2$]\label{lem:grad-expansion-c}
On the unit sphere $\mathbb{S}^2$, with $h=v+zn$, $n=f_0$ and $v\perp n$, we
have the pointwise identity
\begin{align}
|\nabla(zn)|^2
 &= |\nabla z|^2 + 2 z^2. \label{eq:grad-zn-pointwise-c}
\end{align}
In particular, there exists $C<\infty$ such that
\begin{equation}\label{eq:grad-h-error-W22-c}
\biggl|\int_{\mathbb{S}^2}|\nabla h|^2\,dx
       - \int_{\mathbb{S}^2}\bigl(|\nabla z|^2+2z^2\bigr)\,dx\biggr|
\;\le\;
C\,\delta^3.
\end{equation}
\end{lemma}

\begin{proof}
The pointwise identity \eqref{eq:grad-zn-pointwise-c}  follows from the decomposition
$$\nabla(zn)=(\nabla z)n+z\nabla n$$ and the relations $\nabla z\cdot n=0$ together with $|\nabla n|^2=2$
on $\mathbb{S}^2$.
For the integral estimate\eqref{eq:grad-h-error-W22-c}, by the decomposition
\[
\nabla h = \nabla v + \nabla (zn)
\]
and the relations $\nabla_a v\cdot n = -(v,a)$,  we rewrite
\[
|\nabla h|^2 - (|\nabla z|^2+2z^2)
 = |\nabla v|^2+2\nabla v\cdot \nabla(zn)=|\nabla v|^2+ O(|\nabla v|(|h|+|\nabla h|))
\]
Using the bounds on $v$ from Lemma~\ref{lem:v-Lp} (with $p=2$) and the
smallness of $h$, we obtain
$\|v\|_{W^{1,2}}\lesssim\|\nabla h\|_{L^4}^2+\|h\|_{W^{1,2}}^2\lesssim\delta^2$, hence all
terms involving $v$ contribute at most $C\delta^3$ after integration, which
gives \eqref{eq:grad-h-error-W22-c}.
\end{proof}

\begin{lemma}[Quadratic expansion with parameter]\label{lem:quad-form-c}
Under the normalization and smallness assumption
$\delta:=\|h\|_{W^{2,2}(\mathbb{S}^2)}\ll1$, there exists $C<\infty$,
depending on $c$, such that
\begin{equation}\label{eq:final-quad-form-c}
\begin{split}
E[h]
&= -2\int_{\mathbb{S}^2}\Bigl[\Delta(f_0+h)
     + c(f_1\times f_2)\Bigr]\cdot h\,dx\\
&\quad - \int_{\mathbb{S}^2}|\nabla z|^2\,dx
       + 2(c-1)\int_{\mathbb{S}^2}z^2\,dx
       + \widetilde{R}_c[h],
\end{split}
\end{equation}
where the error satisfies
\begin{equation}\label{eq:Rc-error-bound}
|\widetilde{R}_c[h]|\;\le\; C\,\delta^3.
\end{equation}
\end{lemma}
\begin{proof}
It remains to treat the last term
\[
2c\int_{\mathbb{S}^2}(f_1\times f_2)\cdot (f-f_0)\,dx
 \;=\; 2c\int_{\mathbb{S}^2}(f_1\times f_2)\cdot h\,dx .
\]
We estimate it exactly as in Lemma~\ref{lem:z-vs-z2-volume}.  Using the volume
constraint \eqref{eq:volume-diff-zero}, we may rewrite
\[
\int_{\mathbb{S}^2}(f_1\times f_2)\cdot h\,dx
= \int_{\mathbb{S}^2}(e_1\times e_2 - f_1\times f_2)\cdot f_0\,dx .
\]
Expanding $f=f_0+h$ (hence $f_i=e_i+h_i$) yields
\[
e_1\times e_2 - f_1\times f_2
= -(e_1\times h_2 + h_1\times e_2 + h_1\times h_2),
\]
and therefore
\[
2c\int_{\mathbb{S}^2}(f_1\times f_2)\cdot h\,dx
= -2c\int_{\mathbb{S}^2}(e_1\times h_2 + h_1\times e_2 + h_1\times h_2)\cdot f_0\,dx .
\]
By the computation in Lemma~\ref{lem:z-vs-z2-volume}, we get
$$(e_1\times h_2 + h_1\times e_2 + h_1\times h_2)\cdot f_0=\mathrm{div}v+2z+z^2+R',$$
where $R'=O(|\nabla v|(|h|+|\nabla h|))$.
After integration and applying Lemma \ref{lem:z-vs-z2-volume}, we get
\begin{align*}
2c\int_{\mathbb{S}^2}(f_1\times f_2)\cdot h\,dx
= 2c\int_{\mathbb{S}^2} z^2\,dx + O(\delta^3).
\end{align*}
Substituting  this into
\begin{equation}\label{eq:linear-comb-c-compact}
\begin{split}
E[h]
&= -2\int_{\mathbb{S}^2}\Delta(f_0+h)\cdot h\,dx
   - \int_{\mathbb{S}^2}|\nabla h|^2\,dx\\
&= -2\int_{\mathbb{S}^2}\Bigl[\Delta(f_0+h)
     + cf_1\times f_2\Bigr]\cdot h\,dx - \int_{\mathbb{S}^2}|\nabla h|^2\,dx
      + 2c\int_{\mathbb{S}^2}(f_1\times f_2)\cdot h\,dx,
\end{split}
\end{equation}
then applying  Lemma \ref{lem:grad-expansion-c} gives
\eqref{eq:final-quad-form-c} with remainder estimate
\eqref{eq:Rc-error-bound}.
\end{proof}

\begin{proposition}
\label{prop:alexandrov-cmc-control}
Let $c>0$ and assume the normalization, volume constraint, and smallness
assumption $\delta:=\|h\|_{W^{2,2}(\mathbb{S}^2)}\ll1$ as above. Then for any $c<4$, there
exist $C<\infty$ with the following property.

Then
\begin{equation}\label{eq:E-controlled-by-Hc}
E[h]
 \;\le\; C\Bigl(\|H-c\|_{L^2(\Sigma)}^2 + \delta^3\Bigr),
\end{equation}
where $H$ denotes the scalar mean curvature of $\Sigma=f(\mathbb{S}^2)$
(with respect to the unit normal $n+\nu$) and the $L^2$-norm is taken with
respect to $d\mu_f$.
\end{proposition}

\begin{proof}
 For a conformal
immersion $f$ we have
\[
\Delta f = -H(f_1\times f_2),
\]
so
\[
\Delta(f_0+h) + c(f_1\times f_2)
 = (c-H)(f_1\times f_2),
\]
 In particular,
\begin{equation}\label{eq:CMC-defect-pointwise}
\bigl|\Delta(f_0+h) + c(f_1\times f_2)\bigr|
 = \tfrac12|H-c|\,|\nabla f|^2,
\end{equation}
and hence
\begin{equation}\label{eq:CMC-defect-L2}
\int_{\mathbb{S}^2}\bigl|\Delta(f_0+h)
   + c(f_1\times f_2)\bigr|^2\,dx
 = \frac{1}{2}\int_{\Sigma} |H-c|^2|\nabla f|^2\,d\mu_f.
\end{equation}

By Cauchy--Schwarz and \eqref{eq:CMC-defect-pointwise},
\begin{equation}\label{eq:first-term-estimate}
\biggl|
\int_{\mathbb{S}^2}\Bigl[\Delta(f_0+h)
     + c(f_1\times f_2)\Bigr]\cdot h\,dx
\biggr|\le
\Bigl\||H-c||\nabla f|\Bigr\|_{L^2(\Sigma)}\,\|h\|_{L^2(\mathbb{S}^2)}.
\end{equation}
Since $|\nabla f|^2=2e^{2u}$ is uniformly comparable to $1$,
$\||H-c||\nabla f|\|_{L^2(\Sigma)}\leq C\|H-c\|_{L^2(\Sigma)}$.

Next, decomposing $h=v+zn$ and using Lemma~\ref{lem:v-Lp} with $p=2$, we
obtain
\begin{equation}\label{eq:h-L2-vs-z}
\|h\|_{L^2}
 \le \|v\|_{L^2} + \|z\|_{L^2}
 \le C\,\delta^2 + \|z\|_{L^2}.
\end{equation}
Combining \eqref{eq:first-term-estimate} and \eqref{eq:h-L2-vs-z} yields
\begin{equation}\label{eq:first-term-final}
\biggl|
\int_{\mathbb{S}^2}\Bigl[\Delta(f_0+h)
     + c(f_1\times f_2)\Bigr]\cdot h\,dx
\biggr|\le
C\,\|H-c\|_{L^2(\Sigma)}\bigl(\|z\|_{L^2}+\delta^2\bigr).
\end{equation}

By Young's inequality, for any $\varepsilon>0$,
\begin{equation}\label{eq:Young-z}
C\,\|H-c\|_{L^2(\Sigma)}\bigl(\|z\|_{L^2}+\delta^2\bigr)
 \;\le\; C_\varepsilon\,\|H-c\|_{L^2(\Sigma)}^2
       + \varepsilon\,(\|z\|_{L^2}^2+\delta^4).
\end{equation}

We now study the quadratic form in $z$ appearing in
\eqref{eq:final-quad-form-c}:
\begin{equation}\label{eq:z-quadratic-form}
-\int_{\mathbb{S}^2}|\nabla z|^2\,dx
 + 2(c-1)\int_{\mathbb{S}^2}z^2\,dx.
\end{equation}

Decompose $z$ into spherical harmonics.  By
Lemma~\ref{lem:norm-conseq}, the $L^2$--projection of $z$ onto the first
eigenspace (eigenvalue $\lambda_1=2$) is $O(\delta^2)$.  Moreover,
Lemma~\ref{lem:z-vs-z2-volume} implies
\[
\biggl|\int_{\mathbb{S}^2} z\,dx\biggr|
 \le C\int_{\mathbb{S}^2}z^2\,dx + C\delta^3,
\]
so the constant part of $z$ is also of size $O(\delta^2)$.  Thus the
low–frequency component of $z$ is $O(\delta^2)$, and on the orthogonal
complement the Laplacian has spectrum bounded below by $\lambda_2=6$,
yielding
\begin{equation}\label{eq:z-spectrum-lower}
\int_{\mathbb{S}^2}|\nabla z|^2\,dx
 \ge 6\int_{\mathbb{S}^2}z^2\,dx - C\delta^4.
\end{equation}
Plugging \eqref{eq:z-spectrum-lower} into \eqref{eq:z-quadratic-form} we get
\begin{equation}\label{eq:z-quadratic-negative}
-\int_{\mathbb{S}^2}|\nabla z|^2\,dx
 + 2(c-1)\int_{\mathbb{S}^2}z^2\,dx
 \le -(6-2(c-1))\int_{\mathbb{S}^2}z^2\,dx + C\delta^4.
\end{equation}
For $c<4$ we have $6-2(c-1)>0$, so the quadratic form in $z$ is strictly
negative modulo an $O(\delta^4)$ error.  In particular, for $\varepsilon>0$
small and $\delta$ sufficiently small, the $\varepsilon\|z\|_{L^2}^2$ term
in \eqref{eq:Young-z} can be absorbed into the negative contribution
\eqref{eq:z-quadratic-negative}, and the remaining error is bounded by
$C\delta^3$.

Finally, combining \eqref{eq:final-quad-form-c}, \eqref{eq:first-term-final},
\eqref{eq:Young-z}, \eqref{eq:z-quadratic-negative} and the error bound
\eqref{eq:Rc-error-bound}, we obtain
\[
E[h]
 \;\le\; C\,\|H-c\|_{L^2(\Sigma)}^2 + C\,\delta^3
\]
for all $c<4$ and $\delta$ sufficiently small.  This is
exactly \eqref{eq:E-controlled-by-Hc}.
\end{proof}

\begin{corollary}[$W^{1,2}$--control and almost-orthogonality of $z$]
\label{cor:W12-h-and-projections}
Assume the normalization and volume constraint of the previous sections, and
let
\[
h := f-f_0,\qquad h = v+zn,\qquad n=f_0,\qquad
\delta := \|h\|_{W^{2,2}(\mathbb{S}^2)}\ll1.
\]
Let $H$ be the scalar mean curvature of $\Sigma=f(\mathbb{S}^2)$ with respect
to the unit normal $n+\nu$, and let $c>0$ with $c<4$.
Then there exists $C<\infty$ such that, for $\delta$ sufficiently small,
\begin{align}
\|h\|_{W^{1,2}(\mathbb{S}^2)}^2
 &\le C\Bigl(\|H-c\|_{L^2(\Sigma)}^2 + \delta^3\Bigr),
 \label{eq:W12-h-control-c}\\[0.4em]
\biggl|\int_{\mathbb{S}^2} z\,dx\biggr|
 &\le C\Bigl(\|H-c\|_{L^2(\Sigma)}^2 + \delta^3\Bigr),
 \label{eq:int-z-control-c}\\[0.4em]
\biggl|\int_{\mathbb{S}^2} z\,x^i\,dx\biggr|
 &\le C\Bigl(\|H-c\|_{L^2(\Sigma)}^2 + \delta^3\Bigr),
 \qquad i=1,2,3,
 \label{eq:int-zxi-control-c}
\end{align}
or equivalently
\begin{equation}\label{eq:int-zn-vector-control-c}
\biggl|\int_{\mathbb{S}^2} z\,n\,dx\biggr|
 \;\le\; C\Bigl(\|H-c\|_{L^2(\Sigma)}^2 + \delta^3\Bigr).
\end{equation}
\end{corollary}

\begin{proof}
From Proposition~\ref{prop:alexandrov-cmc-control} we have
\begin{equation}\label{eq:E-upper-Hc}
E[h]
 \le C\Bigl(\|H-c\|_{L^2(\Sigma)}^2 + \delta^3\Bigr).
\end{equation}
Combining this with the identity (Lemma \ref{lem:energy-identities-c})
\begin{equation*}
E[h]
 = 4\int_{\mathbb{S}^2} z\,dx
   + \int_{\mathbb{S}^2}|\nabla h|^2\,dx,
\end{equation*}
the refined gradient expansion (Lemma \ref{lem:grad-expansion-c})
and the relation between the average of $z$ and $z^2$(Lemma \ref{lem:z-vs-z2-volume}),
we obtain the quadratic form estimate
\begin{equation}\label{eq:Qz-control}
Q(z)
 := \int_{\mathbb{S}^2}\bigl(|\nabla z|^2-2z^2\bigr)
 \;\le\; C\Bigl(\|H-c\|_{L^2(\Sigma)}^2 + \delta^3\Bigr).
\end{equation}

Using the spectral gap on $\mathbb{S}^2$ together with the normalization
(Lemma~\ref{lem:norm-conseq}) and Lemma~\ref{lem:z-vs-z2-volume}, the
$0$th and $1$st spherical harmonic components of $z$ satisfy
\begin{equation*}
\|P_{\le1} z\|_{L^2} \le C\delta^2,
\end{equation*}
so inserting \eqref{eq:Qz-control} yields
\begin{equation*}
\|z\|_{W^{1,2}(\mathbb{S}^2)}^2
 \le C\Bigl(\|H-c\|_{L^2(\Sigma)}^2 + \delta^3\Bigr).
\end{equation*}
For $v$, Lemma~\ref{lem:v-Lp} (with $p=2$) yields
\begin{equation*}
\|v\|_{W^{1,2}(\mathbb{S}^2)} \le C\delta^2.
\end{equation*}
Thus
\begin{equation}\label{eq:h-W12-final}
\|h\|_{W^{1,2}}^2
 \le C\bigl(\|z\|_{W^{1,2}}^2 + \|v\|_{W^{1,2}}^2\bigr)
 \le C\Bigl(\|H-c\|_{L^2(\Sigma)}^2 + \delta^3\Bigr),
\end{equation}
which establishes \eqref{eq:W12-h-control-c}.

Finally, Lemma~\ref{lem:norm-conseq} and Lemma \ref{lem:z-vs-z2-volume} implies
\begin{equation*}
\biggl|\int_{\mathbb{S}^2} z\,x^i\,dx\biggr|
 \le C\,\|h\|_{W^{1,2}(\mathbb{S}^2)}^2 \quad \text{ and } \quad \biggl|\int_{\mathbb{S}^2} zdx\biggr|\le \int_{\mathbb{S}^2}z^2dx+C\delta^3,
\end{equation*}
and inserting \eqref{eq:h-W12-final} gives \eqref{eq:int-zxi-control-c} and \eqref{eq:int-z-control-c}.
Since $n=(x^1,x^2,x^3)$, \eqref{eq:int-zn-vector-control-c} follows.
\end{proof}

We note that, except for Proposition~\ref{prop:alexandrov-cmc-control},
the estimates in this section do not rely on an a~priori Lipschitz bound
for~$f$ (or an $L^\infty$--bound for the conformal factor~$u$). We denote by $n_f=n+\nu$ the unit normal to $\Sigma$ along $f$, so that $\nu$
measures the normal deviation of $f$ from $f_0$. 
The following lemma gives a quantitative control of the normal error~$\nu$.
As expected from standard computations in differential geometry, we show that~$\nu$ is well approximated by~$-\nabla z$,
up to higher--order terms.
This estimate will be useful in the sequel; its proof will make essential
use of the qualitative $L^\infty$--smallness of~$u$.

\begin{lemma}[Control of $\nu+\nabla z$ in $L^p$]
\label{lem:nu-gradz-Lp-c}
Let $\delta := \|h\|_{W^{2,2}(\mathbb{S}^2)}$. Then for every $p\in[1,\infty)$
there exists a constant $C_p<\infty$ such that
\begin{equation}\label{eq:nu-gradz-Lp-c}
\|\nu+\nabla z\|_{L^p(\mathbb{S}^2)}
 \;\le\; C_p\,\delta^2.
\end{equation}
In particular,
\begin{equation}\label{eq:nu-gradz-L2-c}
\int_{\mathbb{S}^2}|\nu + \nabla z|^2\,dx
 \;\le\; C\,\delta^4.
\end{equation}
\end{lemma}

\begin{proof}

We express the normal variation $\nu$ via the normalized cross product.  In
local coordinates $(y^1,y^2)$ we have
\[
\nu = 2|\nabla(f_0+h)|^{-2}(f_0+h)_1\times(f_0+h)_2
      - (f_0)_1\times(f_0)_2.
\]
Our basic smallness assumption implies that
\[
\bigl||\nabla(f_0+h)|^2 - 2\bigr|\le 1
\]
pointwise. Writing
\[
|\nabla(f_0+h)|^2 = 2+\delta_1,
\]
where $\delta_1 := 2\nabla f_0\cdot\nabla h + |\nabla h|^2=O(|\nabla h|+|\nabla h|^2),$ we have
\[
\frac{2}{|\nabla(f_0+h)|^2}
 = \frac{2}{2+\delta_1}
 = 1 - \frac{\delta_1}{2} + R_1=1-\nabla f_0\cdot\nabla h-\frac{|\nabla h|^2}{2}+R_1,
\]
where $R_1 = \frac{\delta_1^2}{2(2+\delta_1)}$ and hence
\[
|R_1|\le C\,\delta_1^2
 = C\bigl||\nabla(f_0+h)|^2-2\bigr|^2\le C(|\nabla h|^2+|\nabla h|^4).
\]

  Expanding the cross product
then yields
\[
\nu
 = (f_0)_1\times h_2 + h_1\times(f_0)_2
   - (\nabla f_0\cdot\nabla h)n + R_2,
\]
with $|R_2|\le C(|\nabla h|^2+|\nabla h|^4)$.
Writing $h=v+z n$ and using
\[
h_1 = v_1 + z_1 n + z e_1,\qquad
h_2 = v_2 + z_2 n + z e_2,
\]
we obtain
\[
h_1\times(f_0)_2
 = v_1\times e_2 - z_1 e_1 + z n,\qquad
(f_0)_1\times h_2
 = e_1\times v_2 - z_2 e_2 + z n,
\]
thus by $\nabla f_0 \cdot \nabla h=\mathrm{div}h=\mathrm{div}z+2z$, we get
\begin{align}\label{eq:nu-grad-z-pointwise}
\nu=-\nabla z+2z n+O(|\nabla v|)-2zn+R_2=-\nabla z+R_3,  \qquad
|R_3|\le C\bigl(|\nabla v| + |\nabla h|^2 + |\nabla h|^4\bigr).
\end{align}

By Sobolev embedding,  for every $q<\infty$ there exists $C_q$ with
\[
\|\nabla h\|_{L^q(\mathbb{S}^2)}
 \le C_q\,\|h\|_{W^{2,2}(\mathbb{S}^2)}
 \le C_q\,\delta.
\]
From \eqref{eq:nu-grad-z-pointwise} and Lemma \ref{lem:v-Lp},  we get, for $p\in[1,\infty)$,
\[
\begin{split}
\|\nu+\nabla z\|_{L^p}
 &= \|R_3\|_{L^p}
 \le C\Bigl(\|\nabla h\|_{L^{2p}}^2 +\|h\|_{W^{1,2}}^2 +\|\nabla h\|_{L^{4p}}^4\Bigr)\le C_p\bigl(\delta^2 + \delta^4\bigr)
 \le C_p\,\delta^2
\end{split}
\]
for $\delta$ sufficiently small, which proves \eqref{eq:nu-gradz-Lp-c}. In
particular, taking $p=2$ gives \eqref{eq:nu-gradz-L2-c}.

\end{proof}

\section{Controlling the deviation by the CMC defect}

In this subsection we prove that the geometric deviation
\[
\|h\|_{W^{2,2}(\mathbb{S}^2)}+\|u\|_{L^\infty(\mathbb{S}^2)}
\]
is controlled by the $L^2$--oscillation of the mean curvature.  As a
consequence, we prove Theorem \ref{thm:main}.

Recall that the scalar mean curvature $H$ is defined with respect to the
unit normal $n+\nu$, and we denote its average by
\[
\overline H
 := \frac{1}{\mathcal{H}^2(\Sigma)}\int_{\Sigma}H\,d\mu_f.
\]
Under our assumptions we know that $|\overline H-2|\ll1$, so in the
Alexandrov-type estimate of the previous subsection we may (and will) take
$c=\overline H$.

We first show that the average mean curvature is close to $2$, with a
quantitative bound in terms of $H-\overline H$ and $\delta$.

\begin{lemma}[Control of $\overline H-2$]\label{lem:Hbar-2-control}
There exists $C<\infty$ such that, for $\delta:=\|h\|_{W^{2,2}(\mathbb{S}^2)}$ sufficiently small,
\begin{equation}\label{eq:Hbar-2-control}
|\overline H-2|
 \;\le\; C\Bigl(\|H-\overline H\|_{L^2(\Sigma)} ^2+\delta\,\|H-\overline H\|_{L^2(\Sigma)} + \delta^3\Bigr).
\end{equation}
\end{lemma}

\begin{proof}
Let $\varphi := (n+\nu)\cdot f$.  By the divergence theorem (or Minkowski
formula) we have
\[
\int_{\Sigma} H\varphi\,d\mu_f = 2\,\mathcal{H}^2(\Sigma).
\]
We write
\[
H\varphi
 = (H-\overline H)(\varphi-\overline\varphi)
   + \overline H\,\overline\varphi
   + (H-\overline H)\overline\varphi
   + \overline H\,(\varphi-\overline\varphi),
\]
so that
\begin{equation}\label{eq:Hphi-decomp}
2\mathcal{H}^2(\Sigma)=\int_{\Sigma}H\varphi\,d\mu_f
 = \int_{\Sigma}(H-\overline H)(\varphi-\overline\varphi)\,d\mu_f
   + \overline H\,\overline\varphi\,\mathcal{H}^2(\Sigma),
\end{equation}
because the mixed terms integrate to zero.

On the other hand, by definition of $\varphi$ and the divergence theorem in
$\mathbb{R}^3$,
\[
\int_{\Sigma}\varphi\,d\mu_f
 = \int_{\Sigma}(n+\nu)\cdot f\,d\mu_f
 = \int_{\Omega}\operatorname{div}x\,dx = 3|\Omega|
 = 4\pi,
\]
since $|\Omega|=|B_1|=4\pi/3$.  Thus
\[
\overline\varphi
 = \frac{1}{\mathcal{H}^2(\Sigma)}\int_{\Sigma}\varphi\,d\mu_f
 = \frac{4\pi}{\mathcal{H}^2(\Sigma)}.
\]
By the area bound  Proposition \ref{prop:alexandrov-cmc-control},
\[
\bigl|\mathcal{H}^2(\Sigma)-4\pi\bigr|=\frac{1}{2}E(h)
 \le C\Bigl(\|H-\overline H\|_{L^2(\Sigma)}^2 + \delta^3\Bigr),
\]
we obtain
\begin{equation}\label{eq:phi-bar-near-1}
|\overline\varphi-1|
 \le C\Bigl(\|H-\overline H\|_{L^2(\Sigma)}^2 + \delta^3\Bigr).
\end{equation}

We now estimate the fluctuation $\varphi-\overline\varphi$.  Writing
\[
\varphi - 1
 = (n+\nu)\cdot f - n\cdot f_0
 = (n+\nu)\cdot h + \nu\cdot f_0,
\]
and using the bounds on $h$ and $\nu$ (
Lemma~\ref{lem:nu-gradz-Lp-c}) and the bound of the conformal factor, we
have
\[
\|\varphi-1\|_{L^2(\Sigma)} \le C\,\delta.
\]
This implies
\begin{equation}\label{eq:phi-fluctuation}
|\overline\varphi-1|\le C\delta,\qquad|\varphi-\overline\varphi\|_{L^2(\Sigma)} \le C\,\delta.
\end{equation}

Rewriting \eqref{eq:Hphi-decomp}, we have

\[
(2-\overline H\,\overline\varphi)\,\mathcal{H}^2(\Sigma)
 = \int_{\Sigma}(H-\overline H)(\varphi-\overline\varphi)\,d\mu_f.
\]
Hence
\[
|2-\overline H\,\overline\varphi|
 \le \frac{1}{\mathcal{H}^2(\Sigma)}
      \|H-\overline H\|_{L^2(\Sigma)}\,
      \|\varphi-\overline\varphi\|_{L^2(\Sigma)}
 \le C\,\delta\,\|H-\overline H\|_{L^2(\Sigma)},
\]
using \eqref{eq:phi-fluctuation} and the area bounds.

Finally,
\[
\overline H - 2
 = \overline H(1-\overline\varphi)
   + (\overline H\overline\varphi - 2),
\]
so
\[
|\overline H-2|
 \le C|\overline\varphi-1|
    + C\,\delta\,\|H-\overline H\|_{L^2(\Sigma)}
 \le C\Bigl(\|H-\overline H\|_{L^2(\Sigma)}^2+\delta\,\|H-\overline H\|_{L^2(\Sigma)} + \delta^3\Bigr),
\]
where we used \eqref{eq:phi-bar-near-1} in the last step.  This is
\eqref{eq:Hbar-2-control}.
\end{proof}

\subsection{Control of the vectorial defect $\vec H+2x$}

In this subsection we regard $H$ as the mean curvature \emph{vector}
\(\vec H\) (with respect to the unit normal $n+\nu$) and write
\(\vec H = -H(n+\nu)\).  We also denote by $x=f$ the position vector along
$\Sigma$.

\begin{lemma}[Control of $\vec H+2x$ in $L^2$]\label{lem:Hplus2x-L2}
There exists $C<\infty$ such that, for $\delta:=\|h\|_{W^{2,2}(\mathbb{S}^2)}$ sufficiently small,
\begin{equation}\label{eq:Hplus2x-L2}
\|\vec H+2x\|_{L^2(\Sigma)}^2
 \;\le\; C\Bigl(
   \|H-\overline H\|_{L^2(\Sigma)}^2
   + \delta\,\|H-\overline H\|_{L^2(\Sigma)}
   + \delta^3\Bigr).
\end{equation}
\end{lemma}

\begin{proof}
We expand
\[
\int_{\Sigma}|\vec H+2x|^2\,d\mu_f
 = \int_{\Sigma}|\vec H|^2\,d\mu_f
   + 4\int_{\Sigma}\vec H\cdot x\,d\mu_f
   + 4\int_{\Sigma}|x|^2\,d\mu_f.
\]
We treat the three terms separately.

Firstly, by the Minkowski formula,
\[
\int_{\Sigma}\vec H\cdot x\,d\mu_f
 = -\int_{\Sigma}H(x\cdot(n+\nu))\,d\mu_f
 = -2\,\mathcal{H}^2(\Sigma),
\]
so
\[
4\int_{\Sigma}\vec H\cdot x\,d\mu_f
 = -8\,\mathcal{H}^2(\Sigma).
\]

Secondly, we use the elementary decomposition
\[
\int_{\Sigma} (|\vec H|^2-4)\,d\mu_f
 = \int_{\Sigma} (H-\overline H)^2\,d\mu_f
   + (\overline H^2-4)\,\mu_f(\Sigma).
\]
Since $\overline H^2-4=(\overline H+2)(\overline H-2)$ and $|\overline H+2|$ is uniformly bounded for $\delta$ sufficiently small,
Lemma~\ref{lem:Hbar-2-control} yields
\begin{equation}\label{eq:H2-minus4-control}
\biggl|\int_{\Sigma}(|\vec H|^2-4)\,d\mu_f\biggr|
 \;\le\; C\Bigl(
   \|H-\overline H\|_{L^2(\Sigma)}^2
   + \delta\,\|H-\overline H\|_{L^2(\Sigma)}
   + \delta^3\Bigr).
\end{equation}

Lastly, since $x=f=f_0+h$ and $|f_0|=1$, we have
\[
|x|^2 - 1 = |f_0+h|^2-1 = 2f_0\cdot h + |h|^2
           = 2z + |h|^2,
\]
because $f_0=n$ and $h=v+zn$.  Hence
\[
\int_{\Sigma}(|x|^2-1)\,d\mu_f
 = 2\int_{\mathbb{S}^2}ze^{2u}dx + \int_{\Sigma}|h|^2\,d\mu_f.
\]
Using the identity
\[
2(e^{2u}-1)=2\nabla f_0\cdot\nabla h+|\nabla h|^2
\]
and Corollary~\ref{cor:W12-h-and-projections}, we obtain
\begin{equation}\label{eq:int-z-short}
\begin{split}
\Bigl|\int_{\mathbb{S}^2} z e^{2u}\,dx\Bigr|
&\le \Bigl|\int_{\mathbb{S}^2} z(e^{2u}-1)\,dx\Bigr|
     + \Bigl|\int_{\mathbb{S}^2} z\,dx\Bigr| \le C\Bigl(\|H-\overline H\|_{L^2(\Sigma)}^2 + \delta^3\Bigr).
\end{split}
\end{equation}

Similarly, by Corollary~\ref{cor:W12-h-and-projections} ,
\[
\int_{\mathbb{S}^2}|h|^2\,d\mu_f
 \le C\Bigl(\|H-\overline H\|_{L^2(\Sigma)}^2 + \delta^3\Bigr).
\]
Thus
\begin{equation}\label{eq:x2-minus1-control}
\biggl|\int_{\Sigma}(|x|^2-1)\,d\mu_f\biggr|
 \;\le\; C\Bigl(\|H-\overline H\|_{L^2(\Sigma)}^2 + \delta^3\Bigr),
\end{equation}

Putting the three pieces together,
\[
\begin{split}
\int_{\Sigma}|\vec H+2x|^2\,d\mu_f
&= \int_{\Sigma}(|\vec H|^2-4)\,d\mu_f
  + 4\int_{\Sigma}\vec H\cdot x\,d\mu_f
  + 4\int_{\Sigma}(|x|^2-1)\,d\mu_f\\
&= \int_{\Sigma}(|\vec H|^2-4)\,d\mu_f
  + 4\int_{\Sigma}(|x|^2-1)\,d\mu_f.
\end{split}
\]
So \eqref{eq:H2-minus4-control} and \eqref{eq:x2-minus1-control} yields exactly \eqref{eq:Hplus2x-L2}.
\end{proof}

\subsection{Elliptic upgrade to $W^{2,2}$ and control of the conformal factor}

We now use the mean curvature equation to upgrade the $W^{1,2}$--control of
$h$ to a $W^{2,2}$--estimate, and at the same time bound the conformal
factor $u$. This completes the proof of
Theorem~\ref{thm:main}.

Recall that for a conformal immersion $f$ we have
\[
\Delta f = \vec H\,e^{2u},
\]
while for the standard embedding $f_0(x)=x$ we have
\[
\Delta f_0 = -2f_0.
\]
Thus, for $h=f-f_0$,
\begin{equation}\label{eq:Delta-h-equation}
\Delta h
 = \vec H\,e^{2u} + 2f_0.
\end{equation}

\begin{proposition}[CMC stability estimate]\label{prop:CMC-stability}
There exist $\varepsilon_0>0$ and $C<\infty$ such that if
\(\|h\|_{W^{2,2}(\mathbb{S}^2)}+\|u\|_{C^0(\mathbb{S}^2)}\le\varepsilon_0\),
then
\begin{equation}\label{eq:W22uC0-vs-Hdefect}
\|h\|_{W^{2,2}(\mathbb{S}^2)} + \|u\|_{C^0(\mathbb{S}^2)}
 \;\le\; C\,\|H-\overline H\|_{L^2(\Sigma)}.
\end{equation}
\end{proposition}

\begin{proof}
We first estimate the right-hand side of \eqref{eq:Delta-h-equation}.  Write
\[
\vec H\,e^{2u} + 2f_0
 = (\vec H+2f)\,e^{2u} - 2f\,(e^{2u}-1) + 2(f_0-f),
\]
where $f$ is the position vector along $\Sigma$ and we identify $\Sigma$
with $\mathbb{S}^2$ via $f$.  The last term is simply
\[
2(f_0-f) = -2h,
\]
so its $L^2$–norm is controlled by $\|h\|_{L^2}$.

To estimate the factor $e^{2u}-1$, note that the pointwise identity
\[
2(e^{2u}-1)
 = 2\nabla f_0\cdot\nabla h + |\nabla h|^2.
\]
So that
\[
\|e^{2u}-1\|_{L^2(\mathbb{S}^2)}
 \le C\Bigl(\|\nabla h\|_{L^2(\mathbb{S}^2)} + \|\nabla h\|_{L^4(\mathbb{S}^2)}^2\Bigr).
\]
Hence
\[
\|2f\,(e^{2u}-1)\|_{L^2(\mathbb{S}^2)}
 \le 2\,\|e^{2u}-1\|_{L^2(\mathbb{S}^2)}
 \le 2\Bigl(\|\nabla h\|_{L^2} + \|\nabla h\|_{L^4}^2\Bigr),
\]
since $|f|=1$ on $\mathbb{S}^2$.

Finally, for the term $(\vec H+2f)\,e^{2u}$ we use the $L^2$–control of
$\vec H+2f$ from Lemma~\ref{lem:Hplus2x-L2} and the fact that $e^{2u}$ is
uniformly bounded when $\delta$ is small (this follows from the smallness of
$h$ and the conformality relation).  Thus
\[
\|(\vec H+2f)\,e^{2u}\|_{L^2(\mathbb{S}^2)}
 \le C\,\|\vec H+2f\|_{L^2(\Sigma)}.
\]

Collecting these estimates we obtain
\[
\|\vec H\,e^{2u}+2f_0\|_{L^2(\mathbb{S}^2)}
 \le C\Bigl(
   \|\vec H+2f\|_{L^2(\Sigma)}
   + \|h\|_{W^{1,2}(\mathbb{S}^2)}
   + \|\nabla h\|_{L^4(\mathbb{S}^2)}^2\Bigr).
\]
 We may combine this with
Lemma~\ref{lem:Hplus2x-L2} and
Corollary~\ref{cor:W12-h-and-projections} to conclude that
\begin{equation}\label{eq:RHS-Delta-h}
\|\vec H\,e^{2u} + 2f_0\|_{L^2(\mathbb{S}^2)}
 \;\le\; C\Bigl(
   \|H-\overline H\|_{L^2(\Sigma)}
   + \delta^{3/2}+\sqrt{\|H-\overline{H}\|_{L^2(\Sigma)}\delta}\Bigr).
\end{equation}

Now apply the standard elliptic 
\[
\|h\|_{W^{2,2}(\mathbb{S}^2)}
 \le C\Bigl(
   \|\Delta h\|_{L^2(\mathbb{S}^2)}
   + \|h\|_{L^2(\mathbb{S}^2)}\Bigr).
\]
Using \eqref{eq:Delta-h-equation}, \eqref{eq:RHS-Delta-h} and the $L^2$--bound
for $h$ from Corollary~\ref{cor:W12-h-and-projections}, we obtain
\[
\delta=\|h\|_{W^{2,2}(\mathbb{S}^2)}
 \le C\Bigl(
   \|H-\overline H\|_{L^2(\Sigma)}
   + \delta^{3/2}+\sqrt{\|H-\overline{H}\|_{L^2(\Sigma)}\delta}\Bigr)\le C_\varepsilon\|H-\overline{H}\|_{L^2(\Sigma)}+(\varepsilon+C\delta^{\frac{1}{2}})\delta.
\]
For $\delta$ sufficiently small such that $C\delta^{\frac{1}{2}}\le \frac{1}{4}$, taking $\varepsilon=\frac{1}{4}$, the terms involving $\delta$ on the right-hand side  can be absorbed
into the left-hand side, yielding
\begin{align}\label{eq:h w22 estimate}
\|h\|_{W^{2,2}(\mathbb{S}^2)}
 \le C\,\|H-\overline H\|_{L^2(\Sigma)}.
\end{align}

We finally estimate the conformal factor $u$ in terms of the geometric
deviation $h=f-f_0$.  Note that $u$ is already known to be small in $L^\infty$ by the
qualitative theory.

We work in local conformal coordinates on $\mathbb{S}^2$.  Fix a finite
collection of conformal charts
\[
\iota_\alpha : D^2 \longrightarrow \mathbb{S}^2,
\]
with $D^2\subset\mathbb{R}^2$ the unit disc, such that $\{\iota_\alpha(D^2)\}$
covers $\mathbb{S}^2$.
In each chart we have conformal representations
\[
(f_0\circ\iota_\alpha)^*g_{\mathbb{R}^3}
 = e^{2w_{0,\alpha}} g_{D^2},\qquad
(f\circ\iota_\alpha)^*g_{\mathbb{R}^3}
 = e^{2w_\alpha} g_{D^2},
\]
for some $w_{0,\alpha}\in C^\infty(\overline{D^2})$ and
$w_\alpha\in W^{2,2}(D^2)\cap L^\infty(D^2)$.  By definition of $u$ we have
\[
w_\alpha - w_{0,\alpha} = u\circ \iota_\alpha.
\]

For notational convenience we suppress the index $\alpha$ and write, on a
fixed disc $D^2$,
\[
f_0 := f_0\circ\iota,\qquad f := f\circ\iota,\qquad
w_0 := w_{0,\iota},\qquad w := w_\iota,\qquad
u_\iota := u\circ\iota = w-w_0.
\]
We denote by $f_i := \partial_i f$ and $(f_0)_i := \partial_i f_0$
the coordinate derivatives on $D^2$, so that
\[
|f_i| = e^w,\qquad |(f_0)_i| = e^{w_0}.
\]

We now introduce the normalized tangent frames
\[
e_i^0 := \frac{(f_0)_i}{|(f_0)_i|} = e^{-w_0}(f_0)_i,
\qquad
e_i := \frac{f_i}{|f_i|} = e^{-w} f_i,
\]
so that $\{e_1^0,e_2^0\}$ and $\{e_1,e_2\}$ are orthonormal frames for the
pullback metrics of $f_0$ and $f$ on $D^2$, respectively. As is well known,
the Gauss curvature can be written in terms of these frames as
\[
-\Delta w_0 =  \ast(de_1^0\wedge de_2^0),\qquad
-\Delta w     = \ast(de_1   \wedge de_2),
\]
where $\ast$ denotes the Hodge star on $D^2$.

Subtracting the two identities yields an equation for the difference
$u_\iota = w-w_0$:
\begin{equation}\label{eq:Delta-u-local}
-\Delta u_\iota
 = \ast\bigl(d(e_1-e_1^0)\wedge de_2\bigr)
   + \ast\bigl(de_1^0\wedge d(e_2-e_2^0)\bigr)
 \;=: F_\iota.
\end{equation}

Since $|f_i|=e^w$ and $u$ is small in $L^\infty$, we have a uniform lower
bound $|f_i|\ge c>0$ on $D^2$.  Differentiating the definition
$e_i=f_i/|f_i|$ and using the corresponding formula for $e_i^0$, one checks
that
\begin{equation}\label{eq:frame-difference-estimate}
\|e_i-e_i^0\|_{L^2(D^2)}+\|d(e_i-e_i^0)\|_{L^2(D^2)}
 \;\le\; C\,\|h\|_{W^{2,2}(D^2)},\qquad i=1,2.
\end{equation}
On the other hand, $de_i$ and $de_i^0$ are uniformly bounded in $L^2$ on
$D^2$ (the latter because $f_0$ is fixed and smooth, the former because $f$
is conformal and $h$ is small in $W^{2,2}$).
Now, using the extend operator $E:W^{1,2}(D)\to W^{1,2}(\mathbb{R}^2)$ to construct 
\begin{align*}
\tilde{e}_i=E(e_i) \quad \text{ and } \tilde{e}_i^0=E(e_i^0), 
\end{align*}
then $\|\tilde{e}_i\|_{W^{1,2}(\mathbb{R}^2)}+\|\tilde{e}_i^0\|_{W^{1,2}(\mathbb{R}^2)}\le C$ and 
\begin{align*}
\|d(\tilde{e}_i-\tilde{e}_i^0)\|_{L^2(\mathbb{R}^2)}\le C\|e_i-e_i^0\|_{W^{1,2}(D)}\le C\|h\|_{W^{2,2}(D)}. 
\end{align*}

By Wente's inequality\cite{W69}\cite{W80}\cite{MS95}, there exists a unique  solution $v_\iota$ of 
\[
-\Delta v_\iota=\tilde{F}_\iota:=\ast\bigl(d(\tilde{e}_1-\tilde{e}_1^0)\wedge d\tilde{e}_2\bigr)
   + \ast\bigl(d\tilde{e}_1^0\wedge d(\tilde{e}_2-\tilde{e}_2^0)\bigr)
\]
 with $v_\iota(\infty)=0$ and 
\begin{align*}
\|v_\iota\|_{L^\infty(\mathbb{R}^2)}\leq C(\|\tilde{e}_i\|_{W^{1,2}(\mathbb{R}^2)}+\|\tilde{e}_i^0\|_{W^{1,2}(\mathbb{R}^2)})\|d(\tilde{e}_i-\tilde{e}_i^0)\|_{L^2(\mathbb{R}^2)}\leq  C\|h\|_{W^{2,2}(D)}. 
\end{align*}
Since $-\Delta u_\iota = F_\iota=\tilde{F}_{\iota}|_{D}$, we obtain
$$
\Delta (u_\iota-v_\iota)=0 \text{ in } D
$$
and hence 
\begin{align*}
\|u_\iota-v_{\iota}\|_{L^\infty(D')}\le C\|u_\iota-v_\iota\|_{L^1(D)} \text{ for any } D'\subset\subset D. 
\end{align*}
As a conclusion, 
\begin{align*}
\|u_\iota\|_{L^\infty(D')}\le C(\|u\|_{L^1(D)}+\|v\|_{L^\infty(D)})\le C (\|u\|_{L^1(D)}+\|h\|_{W^{2,2}(D)}). 
\end{align*}

Summing over the finitely many charts and using a partition of unity, we
arrive at the global estimate
\begin{equation}\label{eq:u-Linfty-intermediate}
\|u\|_{L^\infty(\mathbb{S}^2)}
 \;\le\; C\, (\|h\|_{W^{2,2}(\mathbb{S}^2)}+\|u\|_{L^1(\mathbb{S}^2)}).
\end{equation}

It remains to estimate $\|u\|_{L^1}$.
Using the conformality relation
$
|\nabla f|^2 = 2e^{2u},|\nabla f_0|^2 = 2,
$
we have
\[
2(e^{2u}-1)
 = |\nabla f|^2 - |\nabla f_0|^2
 = 2\nabla f_0\cdot\nabla h + |\nabla h|^2.
\]
On the other hand, the inequality
\[
\bigl|e^{2u}-1-2u\bigr|\le C u^2
\]
implies
\[
\int_{\mathbb{S}^2} |u|\,dx
 \;\le\; \frac12\int_{\mathbb{S}^2}\Bigl|e^{2u}-1\Bigr|\,dx
         + C\int_{\mathbb{S}^2} u^2\,dx.
\]
Using the identity for $2(e^{2u}-1)$,
\[
\int_{\mathbb{S}^2}\Bigl|e^{2u}-1\Bigr|\,dx
 \le  \int_{\mathbb{S}^2}|\nabla f_0\cdot\nabla h|\,dx
   + \frac12\int_{\mathbb{S}^2}|\nabla h|^2\,dx
\le\; C\|h\|_{W^{1,2}(\mathbb{S}^2)}.
\]
Therefore
\begin{equation}\label{eq:int-u-estimate}
\int_{\mathbb{S}^2} |u|\,dx
 \;\le\; C\|h\|_{W^{1,2}(\mathbb{S}^2)} + C\|u\|_{L^2(\mathbb{S}^2)}^2.
\end{equation}

Inserting \eqref{eq:int-u-estimate} into
\eqref{eq:u-Linfty-intermediate} and using $\|u\|_{L^\infty(\mathbb{S}^2)}\ll 1$, we finally obtain
\begin{equation}\label{eq:u-Linfty-final}
\|u\|_{L^\infty(\mathbb{S}^2)}
 \;\le\; C\,\|h\|_{W^{2,2}(\mathbb{S}^2)}.
\end{equation}
Substituting this into \eqref{eq:h w22 estimate}, we get 
\begin{align*}
\|u\|_{L^\infty(\mathbb{S}^2)}+\|h\|_{W^{2,2}(\mathbb{S}^2)}\le C\|H-\overline{H}\|_{L^2(\Sigma)}. 
\end{align*}
This completes the proof of \eqref{eq:W22uC0-vs-Hdefect}.
\end{proof}

\section{From an area bound to the two--sphere--threshold Willmore bound}
\label{sec:area-implies-willmore}

In this section we show that, in the small--defect regime, an area bound below
the two--bubble threshold forces the two--sphere--threshold Willmore bound
\eqref{eq:Willmore-below-two-spheres}.  The key input is the quantitative
Alexandrov theorem of Julin--Niinikoski~\cite[Theorem~1.2]{JulinNiinikoski23}.

\begin{proposition}[Area bound $\Rightarrow$ two--sphere--threshold Willmore bound]
\label{prop:area-implies-willmore}
Fix $\alpha\in(0,\tfrac12)$ and $\beta>0$.  There exists
$\delta_0=\delta_0(\alpha,\beta)>0$ with the following property.

Let $\Omega\subset\mathbb{R}^3$ be a bounded domain whose boundary
$\Sigma=\partial\Omega$ is a smooth connected embedded surface.  After
rescaling assume that $|\Omega|=\tfrac{4\pi}{3}$.  If
\begin{equation}\label{eq:area-and-defect-assumptions}
\mu_g(\Sigma)\le 4\pi\,\sqrt[3]{2}-\beta,
\qquad
\|H-\overline H\|_{L^2(\Sigma)}\le \delta_0,
\end{equation}
then
\begin{equation}\label{eq:willmore-two-sphere-threshold-conclusion}
\int_{\Sigma}|H|^2\,d\mu_g \le 32\pi(1-\alpha).
\end{equation}
\end{proposition}

\begin{proof}
Argue by contradiction.  Then there exist $\varepsilon_i\downarrow 0$ and
bounded smooth domains $\Omega_i\subset\mathbb{R}^3$ with connected embedded
boundaries $\Sigma_i=\partial\Omega_i$ such that, after rescaling,
\begin{equation}\label{eq:contradiction-setup}
|\Omega_i|=\tfrac{4\pi}{3},\qquad
\mu_{g_i}(\Sigma_i)\le 4\pi\,\sqrt[3]{2}-\beta,\qquad
\|H_i-\overline H_i\|_{L^2(\Sigma_i)}\le \delta_i,
\end{equation}
but
\begin{equation}\label{eq:contradiction-energy}
\int_{\Sigma_i}|H_i|^2\,d\mu_{g_i} \;>\; 32\pi(1-\alpha)\qquad\text{for all }i.
\end{equation}

Set $E_i:=\Omega_i$ and $\lambda_i:=\overline H_i$.  Since $\mu_{g_i}(\Sigma_i)$ is
uniformly bounded and $|E_i|=\tfrac{4\pi}{3}$, the hypotheses
$P(E_i)\le C_0$ and $|E_i|\ge 1/C_0$ in~\cite[Theorem~1.2]{JulinNiinikoski23}
hold for a suitable $C_0=C_0(\delta_0)$.  For $i$ large we also have
$\|H_{E_i}-\lambda_i\|_{L^2(\partial E_i)}=\|H_i-\overline H_i\|_{L^2(\Sigma_i)}\le\delta_i$,
so~\cite[Theorem~1.2]{JulinNiinikoski23} applies (with $n=2$) and yields:
\begin{itemize}
\item a uniform bound $1/C\le \lambda_i\le C$, hence also $R_i:=2/\lambda_i\in[1/C',C']$;
\item an integer $N_i\in\mathbb{N}$ and points $x_{i,1},\dots,x_{i,N_i}$ with
$|x_{i,k}-x_{i,\ell}|\ge 2R_i$ for $k\neq\ell$;
\item writing $F_i:=\bigcup_{k=1}^{N_i} B_{R_i}(x_{i,k})$, one has the
perimeter estimate
\begin{equation}\label{eq:JN-perimeter-approx}
\bigl|\mu_{g_i}(\Sigma_i)-4\pi N_i R_i^2\bigr|
\;\le\; C\,\delta_i^{\frac{1}{64}},
\end{equation}
and a corresponding Hausdorff--type closeness of $E_i$ to $F_i$ (in the sense
stated in~\cite[Theorem~1.2]{JulinNiinikoski23}).
\end{itemize}
The Hausdorff--type closeness implies the volumes of $E_i$ and $F_i$ are
asymptotically the same, hence
\[
|F_i| \;=\; N_i\,\frac{4\pi}{3}\,R_i^3 \;=\; |E_i| + o(1) \;=\; \frac{4\pi}{3}+o(1).
\]
In particular, after passing to a subsequence we may assume $N_i\equiv N$ is
constant and $R_i\to R$, so that $N R^3=1$ and therefore $R=N^{-1/3}$.  Taking
limits in \eqref{eq:JN-perimeter-approx} gives
\[
\lim_{i\to\infty}\mu_{g_i}(\Sigma_i)=4\pi N R^2 = 4\pi N^{1/3}.
\]
Combining with the uniform area bound in \eqref{eq:contradiction-setup} yields
\[
4\pi N^{1/3}\;\le\;4\pi\,\sqrt[3]{2}-\beta,
\]
hence $N<2$ and therefore $N=1$.  Consequently $R=1$ and
$\lambda_i=\overline H_i=2/R_i\to 2$.

With $N=1$, \eqref{eq:JN-perimeter-approx} becomes
\[
\bigl|\mu_{g_i}(\Sigma_i)-4\pi R_i^2\bigr|\;\le\;C\,\delta_i^{\frac{1}{64}},
\qquad R_i=\frac{2}{\overline H_i}.
\]
Equivalently,
\begin{equation}\label{eq:Hbar2-area}
\overline H_i^{\,2}\,\mu_{g_i}(\Sigma_i) \;=\; 16\pi + o(1).
\end{equation}
On the other hand, since $\int_{\Sigma_i}(H_i-\overline H_i)\,d\mu_{g_i}=0$, we have
the exact identity
\begin{equation}\label{eq:H2-splitting}
\int_{\Sigma_i} |H_i|^2\,d\mu_{g_i}
\;=\;
\int_{\Sigma_i}|H_i-\overline H_i|^2\,d\mu_{g_i}
\;+\;
\overline H_i^{\,2}\,\mu_{g_i}(\Sigma_i).
\end{equation}
Using \eqref{eq:contradiction-setup} and \eqref{eq:Hbar2-area} in
\eqref{eq:H2-splitting} we obtain
\[
\int_{\Sigma_i} |H_i|^2\,d\mu_{g_i} \;=\; o(1) + 16\pi \;\longrightarrow\;16\pi,
\]
which contradicts \eqref{eq:contradiction-energy} since
$32\pi(1-\alpha)>16\pi$ for $\alpha\in(0,\tfrac12)$.  This contradiction proves
\eqref{eq:willmore-two-sphere-threshold-conclusion}.
\end{proof}

\medskip

\noindent
As explained in the introduction, Proposition~\ref{prop:area-implies-willmore}
implies the area formulation of Theorem~\ref{thm:main-area}
by combining it with Theorem~\ref{thm:main}.

\end{document}